\documentclass[a4paper,11pt,reqno]{amsart} 
\pdfoutput=1

\usepackage[top=3.5 cm, bottom=3 cm, left=3 cm, right=3 cm]{geometry}

\usepackage[T1]{fontenc}
\usepackage[utf8]{inputenc}

\usepackage{lmodern}

\usepackage{microtype} 
\usepackage{float}
\usepackage{url} 

\usepackage[version=3]{mhchem} 

\usepackage{amsmath}
\usepackage{amssymb}

\usepackage{graphicx}
\usepackage{caption}
\usepackage{subcaption}
\usepackage{float}
\usepackage{wrapfig}

\usepackage{multirow}
\usepackage{array} 

\usepackage{mathrsfs} 

\usepackage{bm}

\usepackage{bbm} 

\usepackage{enumitem}

\usepackage[dvipsnames]{xcolor}

\usepackage{tikz}
\usetikzlibrary{shapes,arrows,shadows}
\usetikzlibrary{positioning}
\usetikzlibrary{cd}
\usepackage{pgfplots}
\pgfplotsset{compat=1.18}
\usepackage{orcidlink}
\usepackage{marvosym}
\usepackage{caption}
\usepackage{color, xcolor}
\definecolor{Egraphcolor}{RGB}{0,160,160} 

\usepackage[version=3]{mhchem}
\usepackage{chemfig} 
\usepackage{quiver}

\usepackage[nameinlink]{cleveref}

    \crefname{example}{Example}{Examples}
    \crefname{theorem}{Theorem}{Theorems} 
    \crefname{lemma}{Lemma}{Lemmas}
    \crefname{proposition}{Proposition}{Propositions}
    \crefname{corollary}{Corollary}{Corollaries} 
    \crefname{conjecture}{Conjecture}{Conjectures} 
    \crefname{definition}{Definition}{Definitions}
    \crefname{remark}{Remark}{Remarks} 
    \crefname{equation}{Eq.}{Eqs.} 
    \crefname{figure}{Fig.}{Figs.}

\usepackage{amsthm}

\newtheorem{theorem}{Theorem}[section]
\newtheorem{lemma}[theorem]{Lemma}
\newtheorem{proposition}[theorem]{Proposition}
\newtheorem{corollary}[theorem]{Corollary}

\theoremstyle{definition}
\newtheorem{definition}[theorem]{Definition}
\newtheorem{example}[theorem]{Example}

\newtheorem{remark}[theorem]{Remark}

\newtheorem*{notation*}{Notation}

\numberwithin{equation}{section}

\usepackage{graphicx}
    \graphicspath{{Figures/}}
\renewcommand{\k}{{\kappa}}
\newcommand{\C}{\mathcal{C}}
\newcommand{\R}{\mathcal{R}}
\newcommand{\E}{\mathcal{E}}
\newcommand{\X}{\mathcal{X}}
\newcommand{\V}{\mathcal{V}}

\newcommand{\vc}[1]{\ensuremath{\begin{pmatrix}#1\end{pmatrix}}}
\newcommand{\imG}{\text{im}~\Gamma}
\newcommand{\cappos}{\operatorname{cap}_{\mathrm{pos}}}

\begin{document}

\title[Multistationarity with inflows and outflows]{Multistationarity in semi-open Phosphorylation-Dephosphorylation Cycles}
\author{Praneet Nandan$^{1*}$}
\thanks{$^1$Laboratoire de Biophysique et Evolution, UMR CNRS-ESPCI 8231 CBI, ESPCI Paris, Université PSL, 10 Rue Vauquelin, Paris, 75005, France, praneet.nandan@espci.fr} 
\author{Beatriz Pascual-Escudero$^{2*}$}
\thanks{$^2$Department of Mathematics and Informatics, Universidad Polit\'ecnica de Madrid, 3 Calle del Prof. Aranguren, Madrid, 28040, Spain, beatriz.pascual@upm.es}
\author{Diego Rojas La Luz$^{3*}$}
\thanks{$^3$Department of Mathematics, University of Wisconsin-Madison, 480 Lincoln Drive, Madison, 53706, Wisconsin, USA, rojaslaluz@wisc.edu}
\thanks{$^*$all authors contributed equally to this work}






\begin{abstract}
    Multistationarity underlies biochemical switching and cellular decision-making. We study how multistationarity in the sequential $n$-site phosphorylation–\\dephosphorylation cycle is affected when only some species are \textit{open}, meaning allowed to exchange with the environment (so-called semi-open networks). Working under mass action kinetics, we obtain two complementary structural results for $n\geq 2$. First, opening any nonempty subset of the substrate species preserves the network’s capacity for nondegenerate multistationarity. Second, opening the enzyme species (both kinase and phosphatase), possibly together with any subset of substrates, always destroys multistationarity. The latter result is proved by a general reduction framework combining the detection of absolute concentration robustness (ACR) with projection onto the remaining species; when the projection produces a monostationary network, the full semi-open system is monostationary. We also illustrate the general method on multi-layer cascade variants and discuss biological implications.
\end{abstract}

\maketitle
\textbf{Keywords:} multistationarity, phosphorylation cycle, deficiency theory, dynamical systems

\textbf{MSC Codes:} 34A34, 37N25, 92C42, 37C25, 80A30

\section{Introduction}
A dynamical system which has multiple steady states is said to have \textit{multistationarity}. This is an important property for biological systems, as it is relevant for decision-making processes during the functioning of the cellular machinery (\cite{laurent-1999,xiong-2003,ozbudak-2004}) and in the context of Darwinian evolution of these systems (\cite{bagley-1991,laurent-1999}). This has led to interest in studying the parameter regions for which certain biologically relevant chemical reaction networks admit or forbid multistationarity(\cite{BIHAN2020367,semiopen,entrapped,joshi-shiu-II,WS}).

\medskip

One example of a biological system which admits multistationarity is the dual phosphorylation network: a substrate with two active sites can be phosphorylated/dephosphorylated through an enzymatic reaction employing a kinase/phosphatase, respectively. This system has been studied using polynomial ordinary differential equations (ODEs) under the simplest assumption of mass action kinetics (\cite{WS,thomson-2009}) and has also become a toy model for developing mathematical machinery focused on understanding multistationarity (\cite{feliu-2020,borri-2021}). More generally, there exist substrates in nature with more than two active sites (e.g. proteins can have >100 sites (\cite{johnson-2004})), and therefore it is relevant to consider the general  (sequential)  $n$-site phosphorylation network given below:

 \begin{small}
 \begin{align*}\notag
S_0 + E \ce{<=>} ES_0 \ce{->}  \dots \ce{<=>}  & ES_{i} \ce{->} S_{i+1}+E  \ce{<=>} \dots \ce{<=>}  ES_{n-1}  \ce{->} S_n+E \\
S_n + F  \ce{<=>} FS_n \ce{->} \dots \ce{<=>}  & FS_{i}  \ce{->} S_{i-1}+F \ce{<=>} \dots \ce{<=>}  FS_1 \quad \ce{->} S_0+F.\\
\end{align*}
\end{small}

This network assumes that, for a substrate that can be phosphorylated at $n$-sites, the process happens at one site at a time and that each encounter modifies one site. The mass action rates depend only on the number of sites already phosphorylated. Under mass action kinetics, a system of polynomial ODEs models the evolution of the concentrations of the species in time. This system involves $3n+3$ variables (one for each species) with $6n$ reaction rates (one for each reaction). The number of steady states of this system has been studied extensively, and multistationarity has been reported for arbitrary $n>1$ (\cite{gunawardena-2005}) with bistability being reported numerically (\cite{salazar-2006}). For $n=1$, there is a unique positive globally asymptotically stable steady state \cite{angeli-2006}. In \cite{WS}, the authors show the existence of reaction rates for which the system has $n$ steady states if $n$ is even and $n+1$ if $n$ is odd. Such high number of steady states is attributed to the fact that a large number of phosphorylation sites compete for the same enzymes. Note that this is a sequential phosphorylation cycle \cite{feliu-2020}, in opposition to the processive one \cite{Conradi2015}, which does not admit multistationarity.

\medskip

Inflow and outflow reactions for a given species represent the exchange of this species with the environment outside the reaction system (through osmosis in case of a membrane, or even just diffusion), as well as other reactions that involve this species and which we do not consider a part of the system of study. This is important from the biological perspective, as the cellular mechanisms of metabolism are often very complex, and involve other physical transformations not tractable with the chemical reaction network formalism. Adding inflows and outflows for a species is what we here refer to as \textit{opening} this species. 
The work in \cite{entrapped} shows that when a reaction network admits multistationarity, by adding inflow and outflow reactions for all species the resulting network will admit multistationarity as well. Some other recent works (\cite{YSE25}) aim at understanding how multistationarity arises in fully open networks. However, this is not always a realistic scenario, as there could be specific species which are \textit{trapped} in the system, unable to exchange freely with the environment (due, for example, to semi-permeable membranes in the cell) \cite{semiopen,entrapped}. Furthermore, some species can be so short-lived (due to high rates of reactions in the network going out from the species) that the out-of-system reactions do not have time to change their concentrations significantly.

\medskip

In this article, we address the question of multistationarity in the case when only some of the reacting species are exchanged with the environment. Apart from the consequences mentioned before concerning applications, the results give us insight regarding the role that different species play in creating a highly multistationary system. Furthermore, they show a way to develop  more general mathematical tools to check for multistationarity in reaction networks which are similar to the phosphorylation cycle. For these results, \Cref{sec:preliminaries} introduces the necessary background, notation and tools. 

\medskip

In Section \ref{sec:substrates}, we show that if any of the species classified as substrate is exchanged with the environment, while the other species are either trapped or get converted into substrates quickly, multistationarity in the system is conserved. These substrates are often molecules which are important for the functioning of biological systems and are undergoing a phosphorylation process to get used up in other processes after their transformation, hence they are involved in other reactions outside the cycle. This allows the phosphorylation cycle to be coupled to these other reactions without losing its multistationarity.

\medskip

In Section \ref{sec:enzymes} we develop a reduction method, that can be applied to the phosphorylation cycle if 
the species classified as enzymes are open. In this system, we find that by opening both enzymes multistationarity is destroyed. This implies that allowing fast exchange of enzymes with the environment can be used as a “switch” to turn a highly multistationary system into a robust monostationary system. Such switches are important for evolutionary selection, as the ability to display different behavior based on environment allows survival  (\cite{balaban-2004,visco-2010}) and adaptability (\cite{blake-2006,kashiwagi-2006}). We also show that this can be applied to a cascade type process, yielding similar results. We are able to showcase part of the rich behavior such multistationary systems are capable of displaying when they interact with other systems and their environment. The reduction formalism can be applied to other more general reaction networks too, which will be the topic of a forthcoming paper, and can have consequences beyond the systems in play.

\section{Preliminaries and General Toolbox}\label{sec:preliminaries}

\subsection{Notation and Definitions}

We will write $\mathbb{R}_{>0}$ for the set of strictly positive real numbers, and $\mathbb{R}_{\ge0}$ for the set of nonnegative real numbers.

\begin{definition}
A \textit{reaction network} can be represented by a directed graph, where each arrow stands for an interaction between two complexes: the source or reactant, and the product. Each of these complexes is a linear combination of the species involved with nonnegative integer coefficients. In other words, a reaction network is a triplet $G=(\X,\C,\R)$ where 
\begin{enumerate}
    \item $\X=\left\{ X_1,\ldots ,X_n\right\}$ is the set of \textit{species},
    \item $\R =\left\{ \sum_{i=1}^n y_{ij} X_i \stackrel{\k _j}{\rightarrow} \sum_{i=1}^n y_{ij}' X_i\right\}_{j=1}^{r}$  for some $y_{ij},y_{ij}'\in \mathbb{Z}_{\geq 0}$ is the set of reactions, each being an edge between the \textit{source} and the \textit{product} of such reaction, and labeled by the corresponding \textit{reaction rate} $\k _j\in \mathbb{R}_{>0}$ and
    \item $\C=\left\{ \sum_{i=1}^n y_{ij} X_i\right\}_{j=1}^r\cup\left\{ \sum_{i=1}^n y'_{ij} X_i\right\}_{j=1}^r$ is the set of \textit{complexes}, appearing as source or product in some reaction. 
\end{enumerate}
\end{definition}

We identify each complex $\sum_{i=1}^n y_i X_i$ with the vector of coefficients $y=(y_1,\ldots ,y_n)\in \mathbb{Z}^n_{\geq 0}$, and write $y\to y'$ for the reaction between the corresponding complexes. This identifies reaction networks with Euclidean embedded graphs (E-graphs). 

\begin{definition}\label{def:embedded_graph}[\cite{Craciun2019}]
    An \textit{Euclidean embedded graph (E-graph)} is a finite simple directed graph $(V, E)$, where $V \subseteq \mathbb{Z}_{\geq 0}^n$ is the set of vertices, and $E \subseteq V \times V$ is the set of directed edges.
\end{definition}

Thus, $\C$ is in correspondence with the set of vertices $V$, and $\R$ is in correspondence with the set of edges $E$. This perspective will be useful when thinking about projected networks (see \cref{def:projected_network}).

\begin{example}\label[example]{ex:E-graph}
Consider the following reaction network $G=(\X, \C, \R)$:

\begin{small}
\begin{align}\label{eq:E-graph}
\begin{aligned}
Y \ce{->} X \ce{->} Z,\\
2Z \ce{->} Y + Z,
\end{aligned}
\end{align}
\end{small}

where $\X = \{X, Y, Z\}$, $\C = \{X, Y, Z, 2Z, Y + Z\}$, and the set of reactions $\R$ is the one shown. This can be represented as a E-graph embedded into 3-dimensional space:

\begin{align}\label{tikz:E-graph}
\begin{aligned}
\begin{tikzpicture}[scale=1.2]

\begin{axis}[view={105}{30}, axis line style=white, width=5.5cm, height=5.5cm, ticks=none, xmin=0, xmax=1.9, ymin=0, ymax=1.9, zmin=0, zmax=2.5]

    \tikzset{bullet/.style={inner sep=1pt,outer sep=1.5pt,draw,fill,Egraphcolor,circle}};
    \tikzset{myarrow/.style={arrows={-stealth},thick,Egraphcolor}};
    \tikzset{mygrid/.style={very thin,gray!50}};
    \tikzset{myaxis/.style={thick,gray}};    
    \newcommand{\ab}{2.5};    
 
    \draw[mygrid] (1,0,\ab) -- (1,0,0) -- (1,\ab,0);
    \draw[mygrid] (2,0,\ab) -- (2,0,0) -- (2,\ab,0);
    \draw[mygrid] (0,1,\ab) -- (0,1,0) -- (\ab,1,0);
    \draw[mygrid] (0,2,\ab) -- (0,2,0) -- (\ab,2,0);
    \draw[mygrid] (0,\ab,1) -- (0,0,1) -- (\ab,0,1);
    \draw[mygrid] (0,\ab,2) -- (0,0,2) -- (\ab,0,2);
        
    \draw[myaxis] (0,0,0) -- (\ab,0,0);    
    \draw[myaxis] (0,0,0) -- (0,\ab,0);    
    \draw[myaxis] (0,0,0) -- (0,0,\ab);    
    
    \node[bullet] (X) at (1,0,0) {};
    \node[bullet] (Y) at (0,1,0) {};
    \node[bullet] (Z) at (0,0,1) {};
    \node[bullet] (2Z) at (0,0,2) {};
    \node[bullet] (Y+Z) at (0,1,1) {};
    
    \node [below right] at (X) {\small $\mathsf{X}$};
    \node [above right] at (Y) {\small $\mathsf{Y}$};
    \node [right] at (Z) {\small $\mathsf{Z}$};
    \node [above right] at (2Z) {\small $\mathsf{2Z}$};
    \node [above right] at (Y+Z) {\small $\mathsf{Y+Z}$};
    \node[right] at (0,1.25,2.25) {$G$};
    
    \draw[myarrow,transform canvas={xshift=-1pt,yshift=0pt}]  (Y) to node[xshift=0pt,yshift=-3pt] {} (X);
    \draw[myarrow,transform canvas={xshift=0pt,yshift=1pt}] (X) to node[xshift=0pt,yshift=-5pt] {} (Z);
    
    \draw[myarrow,transform canvas={xshift=1pt,yshift=0.5pt}] (2Z) to node[xshift=12pt,yshift=4pt] {} (Y+Z);
\end{axis}

\draw[rounded corners,lightgray] (current bounding box.south west) rectangle (current bounding box.north east);

\end{tikzpicture}
\end{aligned}
\end{align}
\end{example}

We will also use the following graph-theoretic property of some reaction networks:

\begin{definition}\label[definition]{def:weakly_reversible}
    We say that $G$ is \textit{weakly reversible} if any edge is part of an oriented cycle in $G$. Equivalently, any connected component of G is \textit{strongly connected}, where by a strongly connected directed graph we mean one where there exists a directed path between every pair of vertices. That is to say, for any $y\to y'\in \R$, there is a path $y'\to \ldots \to y$ given by reactions in $\R$.
\end{definition}

According to \cite{joshi-shiu-II}, a reaction network is a CFSTR (continuous-flow stirred-tank reactor) if, for any species $X_i$ in the network, the \textit{outflow} reaction $X_i \ce{->} 0$ is in the network. A CFSTR is fully open if for all species the pair $0 \ce{<=>} X_i$ is in the network, that is, there are both outflow and \textit{inflow} reactions $0 \ce{->} X_i$ for all species.

\begin{definition}
    We say that a species $X \in \X$ is \textit{closed in} $G$ if none of the reactions $0 \ce{->} X$ and $X \ce{->} 0$ are in $\R$. In contrast, we say that $X$ is \textit{open in} $G$ if the reactions $0 \ce{->} X$ and $X \ce{->} 0$ are both in $\R$. We say that $X$ is \textit{partially open in} $G$ if exactly one of the reactions $0 \ce{->} X$ and $X \ce{->} 0$ is in $\R$. A fully open CFSTR would be then a network in which all species are open. We say that $G$ is a \textit{semi-open network} if some species are open in $G$, but not all.
\end{definition}

Given a reaction network $G=(\X,\C,\R)$ one can model the evolution of the species' abundances through time by a system of ordinary differential equations of the form:
\begin{equation}\label{eq:general_kinetics}
    \dfrac{dx_i}{dt} = f_i(x_1,\dots,x_n) = \sum_{y \to y'\in \R} \nu_{y\to y'}(x(t)) (y'_i - y_i), \quad i=1,\dots,n,
\end{equation}
where $x(t) = (x_1(t),\dots,x_n(t))$, $x_i$ stands for the concentration of species $X_i$, $f_i$ is its \textit{species formation function}, and $(y'_i - y_i)$ represents the net production of species $X_i \in \X$ resulting from the reaction $y\to y' \in \R$. We call a choice of continuously differentiable functions $\nu_{y\to y'}$ a choice of \textit{kinetics}. We now construct the \textit{stoichiometric matrix }$\Gamma \in \mathbb{R}^{n\times r}$, where the $i$-th column is the vector $y' - y$ if the $i$-th reaction is given by $y\to y'$. We call the space generated by the columns of $\Gamma$ the \textit{stoichiometric subspace} $\imG$. We denote by $\nu(x)$ the vector with entries $\nu_{y\to y'}(x)$. Then we can rewrite \eqref{eq:general_kinetics} as:
\begin{equation}\label{eq:stoichiometric_general}
    \dfrac{dx}{dt} = f(x) = \Gamma \nu(x),
\end{equation}
where $f=(f_1,\ldots,f_n)$.

We will be dealing with a widespread choice of kinetics called \textit{mass action kinetics}, in which the rate of each reaction is proportional to the product of the concentrations of the species in the reactant. Thus, given a choice $\kappa\in \mathbb{R}^r_{>0}$ of reaction rates for $G$, the associated mass action kinetics are defined by monomials of the form:
\[\nu_{y\to y'}(x) = \kappa_{y\to y'} x_1^{y_1}\cdots x_n^{y_n},\]
and the corresponding mass action differential equations for the \textit{mass action system} $(G,\k)$ can be written as
\begin{equation}
    \dfrac{dx_i}{dt} = f_i(x_1,\dots,x_n) = \sum_{y \to y'\in \R} \kappa_{y\to y'} x_1^{y_1}\cdots x_n^{y_n} (y'_i - y_i), \quad i=1,\dots,n
\end{equation}

Those concentrations $x=(x_1,\ldots ,x_n)\in \mathbb{R}^n_{>0}$ where $f_1(x)=\ldots =f_n(x)=0$ are the \textit{positive steady states} of the mass action system. We will denote by $\V(G,\k^*)$ the set of positive steady states of the mass action system $(G,\k)$ with fixed reaction rates $\k^*$.

Note that if $w=(w_1,\ldots ,w_n)\in \mathbb{R}^n$ is a vector in the left kernel of the stoichiometric matrix $\Gamma$, then
\begin{equation}
    w \cdot \dfrac{dx}{dt} = w\cdot \Gamma\nu(x) = 0,
\end{equation}
and $w_1 x_1 + \dots + w_nx_n$ will be constant in time. Thus, we call $w$ a \textit{conservation law}. Assuming $\mathrm{rank}(\Gamma)=n-d$, conservation laws define \textit{positive stoichiometric compatibility classes} (which we may shorten by \textit{stoichiometric classes}) as
$$N_{T}=\left\{ x\in \mathbb{R}^n_{>0}: Wx=T\right\} $$
where $W\in \mathbb{R}^{d\times n}$ is a matrix whose rows form a basis of all conservation laws and $T\in \mathbb{R}^d$ is a vector of \textit{total amounts}. Each positive stoichiometric compatibility class is a subset of positive concentrations which give the same vector of total amounts. For a fixed vector $T^*$ of total amounts, we will write  $\V(G,\k)\cap N_{T^*}$ for the set of positive steady states in such stoichiometric class. Note that we can think of $W$ as the matrix generating the linear space given by the (left) kernel of the stoichiometric matrix, $\ker \Gamma^{t}$. Thus the $N_{T}$ are just parallel translations of the stoichiometric subspace intersected with $\mathbb{R}^n_{>0}$. 

\begin{remark}\label[remark]{rmk:same_compatibility}
    Given $x, x'\in\mathbb{R}^n_{>0}$, they are both in $N_T$ for some $T\in \mathbb{R}^d$ if, and only if, $x' - x \in \imG$ is in the stoichiometric subspace. Indeed, this last condition is equivalent to $W(x' - x) = 0$, so if $x\in N_T$ and $W(x' - x) = 0$, then $Wx' = Wx = T$, and $x'\in N_T$. Similarly, if $x$ and $x'\in N_T$, then $W(x' - x) = 0$.
\end{remark}

\begin{example}
    Consider the reaction network $G$ given by \eqref{eq:E-graph} from \cref{ex:E-graph}. The mass action system associated to $G$ is:
    \begin{equation}
        \dfrac{dx}{dt} = \kappa_{Y\to X}~ x_Y \vc{1\\-1\\0} + \kappa_{X \to Z}~ x_X \vc{-1\\0\\1} + \kappa_{2Z\to Y+Z}~x_Z^2 \vc{0\\1\\-1}
    \end{equation}
    where $x_X, x_Y$ and $x_Z$ represent the concentrations of $X, Y$ and $Z$ respectively. Thus, its stoichiometric matrix is
    \[\Gamma = \begin{pmatrix}
        1 & -1 & 0\\
        -1 & 0 & 1\\
        0 & 1 & -1
    \end{pmatrix},\]
    and we see that there exists a linear conservation law $w = (1, 1, 1)$, meaning that $x_X + x_Y + x_Z = T$, the total concentration is conserved.
\end{example}

The following technical definition will be necessary for the results in \Cref{sec:enzymes}: 

\begin{definition}\label[definition]{def:independently_conserved}
      Let $G=(\X,\C,\R)$ be a network and consider a subset $\E = \{ X_i \}_{i=1,\dots, k} \subset \X$ of species. We say that $\E$ is \textit{independently conserved} in $G$ if there is a set of conservation laws $\{ L_1,\ldots ,L_k\}\subset \mathbb{R}^n$ such that for $i, j=1,\ldots ,k$, the $i$-th entry is nonzero in $L_i$ and zero in every other $L_j$. That is, if $X_i$ appears in $L_i$ but not in $L_j$ when $j\neq i$.
\end{definition}

\begin{remark}
    Roughly speaking, systems with catalysts (or enzymes) satisfy this, as the total amount of a given enzyme will result in a conservation law for it, which will include the enzyme and all intermediate complexes of which the enzyme is part. It is the case, for instance, of the $n$-site double phosphorylation cycle (see \Cref{sec:n-site}) and phosphorylation cascades (see \Cref{sec:cascade_monost}).
\end{remark}

\begin{remark}\label[remark]{rem:indep_conserved-invertible}
    Note that, if $\E = \{ X_i \}_{i=1,\dots, k} \subset \X$ are independently conserved, then there are $k$ linearly independent conservation laws (the network has rank at most $n-k$). By reordering species appropriately we will have a block matrix
    $$W=\left(\begin{array}{cc}
       Id_k  &  *\\
        0_{(n-k)\times k} & *
    \end{array}\right)$$
    where the first $k$ rows correspond to the $L_1,\ldots ,L_k$ from \Cref{def:independently_conserved}.
    
    The converse is also true and then, in particular, for a network of rank $n-d$ we can always choose a subset of $d$ independently conserved species in $\X$, as Gaussian elimination will allow us to obtain such a matrix, if the columns (hence the species) are reordered appropriately.
\end{remark}

A steady state $x^*$ is \textit{nondegenerate} if $\ker\left(J(G,\k)(x^*)\right)\cap \imG=\{0\}$, where 
$$J(G,\k)=J_{f,x}=\left( \begin{array}{c} \frac{\partial f}{\partial x}\\ W
\end{array}\right)$$
is the Jacobian matrix of $f$ with respect to $x$ and $\imG=\{y: Wy=0\}$ is the \textit{stoichiometric subspace}. Equivalently, a steady state $x^*$ is nondegenerate if the Jacobian matrix of the (square) system $\{ f(x)=0,Wx=T\}$ has full rank $n$ (see \cite[Section 6]{WF13}). Note that this does not depend on the particular choice of $W$.

\begin{definition}[Positive steady-state capacity {\cite[Section 2.2]{joshi-shiu-small}}]\label{def:positive-capacity}
Under mass action kinetics, we say that a reaction network $G$ \emph{admits $k$ positive steady states} if, for some choice of positive rate constants $\k^*\in\mathbb{R}^r_{>0}$ and some nonempty stoichiometric compatibility class $N_{T^*}$, the set $\V(G,\k^*)\cap N_{T^*}$ contains exactly $k$ distinct points. Here $k=\infty$ means that this set is infinite. The \emph{positive steady-state capacity} of $G$ is
\[
\cappos(G)
:=
\max\left\{
k\in\mathbb{Z}_{\geq 0}\cup\{\infty\}
\;\middle|\;
\begin{array}{l}
G \text{ admits } k \text{ positive steady states in some}\\
\text{stoichiometric compatibility class}
\end{array}
\right\}.
\]
The network $G$ admits \emph{multistationarity} if
\[
\cappos(G)\geq 2.
\]
\end{definition}

Thus, admitting multistationarity means that there is a choice of positive rate constants for which at least two distinct positive steady states lie in the same stoichiometric compatibility class. If two such steady states can be found that are nondegenerate, we say that the network \textit{admits nondegenerate multistationarity}. The quantity $\cappos(G)$ counts all distinct positive steady states, whether degenerate or nondegenerate. We will say that a network with mass action kinetics is \textit{monostationary} if it does not admit multistationarity, that is, if $\cappos(G)\leq 1$. This convention includes networks with no positive steady states and does not assert existence in every stoichiometric compatibility class.

\subsection{Network modifications}\label{sec:network_modifications}

In order to study the capacity for multistationarity of networks when adding certain flow (inflow/outflow) reactions, let us define the following modifications that one can perform on a network.

\begin{definition}
    Let $G=(\X,\C,\R)$ be a network and consider a subset $\E \subset \X$ of species which are closed in $G$. We define a new network $G_{\E\rightleftharpoons 0}$ as the network given by $(\X, \C\cup \{0\}\cup\E,\R')$, where $\R'=\R \cup \left\{ X\rightleftharpoons 0\right\}_{X\in \E}$ and call it the \textit{open network on $\E$} in $G$.
\end{definition}

\begin{definition}\label[definition]{def:projected_network}
     Let $G=(\X,\C,\R)$ be a network and $\E \subset \X$ a subset of species. Consider the network  $G_{-\E}$ obtained from $G$ performing the following operations: For each $X\in\E$,
    \begin{enumerate}
        \item remove $X$ from each complex, replacing the complex $X$ with the $0$ complex (if $X$ is a complex of $G$),
    
        \item remove any self-loops.
    \end{enumerate}
The species of $G_{-\E}$ will be $\X\setminus \E$.
    We call this modification the \textit{projected network} \textit{on the complement of} $\E$, following \cite{doi:10.1137/110840509}. The name comes from the fact that we can think of $G_{-\E}$ as a projection of the complexes and reaction vectors onto the complement of $\E$ if one considers the network as an embedded graph (see Definition \ref{def:embedded_graph}). 
    
    Let $z\in\mathbb{R}^n_{\geq 0}$ be a steady state of the mass action dynamical system given by $G$ with certain reaction rates. Then, the projection of the steady state $z$ onto $\E$, $z_{\E}$, is defined by projecting $z$ over the species given by $\E$.
\end{definition}

\begin{remark}\label[remark]{rmk:projected_stoich}
    Recall that the stoichiometric subspace of $G$, $\imG$, is given by the span of $y' - y$, for $y\to y'\in \R$. 
    Note that any reaction $y\to y'\in \R_{-\E}$ must come from the projection of a reaction $\overline{y} \to \overline{y}' \in \R$ over $\X-\E$, seen as an E-graph. That is, $\pi(\overline{y})=y$ and $\pi(\overline{y}')=y'$, with $\pi:\mathbb{R}^{|\X|}\longrightarrow \mathbb{R}^{|\X\setminus \E|}$. As a consequence, $\imG_{-\E}$ is given by the projection of $\imG$ over $\X-\E$.
\end{remark}

\begin{example}\label[example]{ex:projection}
    Continuing \cref{ex:E-graph}, consider the projection of the reaction network $G$ given by \eqref{eq:E-graph} over the complement of $\E=\{Z\}$, that is $G_{-Z}$, which is given by:

    \begin{small}
    \begin{align}\label{eq:G_E}
    \begin{aligned}
    Y \ce{->} X \ce{->} 0,\\
    0 \ce{->} Y.
    \end{aligned}
    \end{align}
    \end{small}

    This reaction network can be seen as a literal projection of the E-graph given in \eqref{tikz:E-graph} over the $XY$ plane:

    \begin{align}\label{tikz:projection}
    \begin{aligned}
    \begin{tikzpicture}[scale=1.2]
\tikzset{mybullet/.style={inner sep=1.2pt,outer sep=4pt,draw,fill,Egraphcolor,circle}};
\tikzset{myarrow/.style={arrows={-stealth},thick,Egraphcolor}};

\begin{axis}[view={105}{30}, axis line style=white, width=5.5cm, height=5.5cm, ticks=none, xmin=0, xmax=1.9, ymin=0, ymax=1.9, zmin=0, zmax=2.5]

    \tikzset{bullet/.style={inner sep=1pt,outer sep=1.5pt,draw,fill,Egraphcolor,circle}};
    
    \tikzset{mygrid/.style={very thin,gray!50}};
    \tikzset{myaxis/.style={thick,gray}};    
    \newcommand{\ab}{2.5};    
 
    \draw[mygrid] (1,0,\ab) -- (1,0,0) -- (1,\ab,0);
    \draw[mygrid] (2,0,\ab) -- (2,0,0) -- (2,\ab,0);
    \draw[mygrid] (0,1,\ab) -- (0,1,0) -- (\ab,1,0);
    \draw[mygrid] (0,2,\ab) -- (0,2,0) -- (\ab,2,0);
    \draw[mygrid] (0,\ab,1) -- (0,0,1) -- (\ab,0,1);
    \draw[mygrid] (0,\ab,2) -- (0,0,2) -- (\ab,0,2);
        
    \draw[myaxis] (0,0,0) -- (\ab,0,0);    
    \draw[myaxis] (0,0,0) -- (0,\ab,0);    
    \draw[myaxis] (0,0,0) -- (0,0,\ab);    

    \node[bullet] (0) at (0,0,0) {};
    \node[bullet] (X) at (1,0,0) {};
    \node[bullet] (Y) at (0,1,0) {};
    \node[] (Z) at (0,0,1) {};

    \node [below right] at (0) {\small $\mathsf{0}$};
    \node [below right] at (X) {\small $\mathsf{X}$};
    \node [above right] at (Y) {\small $\mathsf{Y}$};
    \node [above right] at (Z) {\small $\mathsf{Z}$};
    \node[right] at (0,1.25,2.25) {$G_{-Z}$};
    
    \draw[myarrow,transform canvas={xshift=-1pt,yshift=0pt}]  (Y) to node[xshift=0pt,yshift=-3pt] {} (X);
    \draw[myarrow,transform canvas={xshift=0pt,yshift=1pt}] (X) to node[xshift=0pt,yshift=-5pt] {} (0);
    
    \draw[myarrow,transform canvas={xshift=1pt,yshift=0.5pt}] (0) to node[xshift=12pt,yshift=4pt] {} (Y);
\end{axis}

\begin{scope}[shift={(5,0.7)}, scale=2.5]

\draw [step=1, gray, very thin] (0,0) grid (1.25,1.25);
\draw [ -, black] (0,0)--(1.25,0);
\draw [ -, black] (0,0)--(0,1.25);

\node[mybullet]  (0) at (0,0) {};
\node[mybullet]  (X) at (1,0) {};
\node[mybullet]  (Y) at (0,1) {};

\node [below right] at (0) {\small $\mathsf{0}$};
\node [below right] at (X) {\small $\mathsf{X}$};
\node [above right] at (Y) {\small $\mathsf{Y}$};

\draw[myarrow,transform canvas={yscale=2, yshift=-11pt}]  (X) to node[above] {} (0);
\draw[myarrow,transform canvas={xscale=2, xshift=-84pt}]  (0) to node[right] {} (Y);
\draw[myarrow,transform canvas={xshift=0pt, yshift=0pt}]  (Y) to node[pos=0.65, below] {} (X);

\node[right] at (1.1,1.2) {$G_{-Z}$};
\node at (1.75,0.5) {\phantom{0}};

\end{scope}

\draw[rounded corners,lightgray] (current bounding box.south west) rectangle (current bounding box.north east);

\end{tikzpicture}
    \end{aligned}
    \end{align}

    Note that its stoichiometric subspace is $\mathbb{R}^2$, and thus no conservation laws are left.
\end{example}

\begin{remark}
    Note that $G_{-\E}$ is an embedded network of $G$ and of $G_{E\rightleftharpoons 0}$, according to the definition in \cite{joshi-shiu-II}. It is also called a reduced reaction network in \cite{Anderson_GAC}.
\end{remark}

A third modification defined below, where a second network is involved, will be necessary for the coming results in \cref{sec:enzymes}:

\begin{definition}
     Let $G=(\X,\C,\R)$ be a network and consider a subset $\E \subset \X$ of species, and any second network $H=(\E,\C',\R')$ involving these species. We define a new network from them, given as $G\cup H=(\X,\C\cup \C',\R\cup \R')$. 
\end{definition}

 \cite{Intermediates-Feliu_Wiuf} characterizes dynamical properties of networks on the inclusion of intermediate species. An \textit{intermediate} in this case, is defined as a species in a reaction network that is produced in at least one reaction, consumed in at least one reaction and it cannot be part of any other complex. The species we later classify in this paper as \textit{Intermediates} follow this definition. They define a network extension model as follows
\begin{definition}\label{extension-model}
    Taking a reaction network $G=(\X,\C,\R)$ (called the \textit{core model} in \cite{Intermediates-Feliu_Wiuf}), an \textit{extension model} is a reaction network $G'=(\X',\C',\R')$, obtained by adding the intermediates $Y\in\mathcal{Y}$ such that:
    \begin{itemize}
        \item $\C'=\C\cup\mathcal{Y}$ (and a similar relation for $\X'$)
        \item All reactions are between two complexes in $\C$, one complex in $\C$ and an intermediate in $\mathcal{Y}$, or two intermediates in $\mathcal{Y}$
        \item $G$ is obtained from $G'$ by collapsing all reaction paths of the form $X_i\rightarrow Y_1\rightarrow \dots \rightarrow Y_k\rightarrow X_j$, where $Y_l\in \mathcal{Y}$ into a single reaction $X_i \rightarrow X_j \in \R$
        \item All intermediates $Y\in\mathcal{Y}$ are part of some reaction path of the form $X_i\rightarrow Y_1\rightarrow \dots \rightarrow Y_k\rightarrow X_j$,
    \end{itemize}
\end{definition}
\begin{theorem}[Extension model obtained as defined above conserves non-degeneracy. See \textit{Theorem 5.1} in \cite{Intermediates-Feliu_Wiuf}] \label{extension-model-multist}
If the core model $G=(\X,\C,\R)$ has $N$ non-degenerate positive steady states for some rate constants in some stoichiometric class, then any extension model as defined above, that realizes the rate constants has at least $N$ corresponding non-degenerate positive steady states for some rate constants in some stoichiometric class.
    
\end{theorem}

\cite{Banaji2018} investigates modifications to a chemical reaction network that preserve nondegenerate multistationarity. Out of those results, we state here two Theorems which we will use later in \Cref{thm:induction}.
\begin{theorem}[Adding intermediate complexes involving new species, see \textit{Theorem 6}  and \textit{Remark 4.6} in \cite{Banaji2018}]\label{thm:interb}
    Let $G=(\X,\C,\R)$ be a network, let $Y$ be a list of $k$ new species, and let $G'$ arise from $G$ by selecting a subset $\R'=\{R_{t_1},\ldots,R_{t_m}\}\subset \R$ of  $m$ reactions and replacing each     $$
        R_{t_i}:\sum_{j=1}^n{a_{t_ij}X_j}\rightarrow \sum_{j=1}^n{b_{t_ij}X_j},\quad i\in \{1,\ldots ,m\}$$ 
    by a chain
$$\sum_{j=1}^n{a_{t_ij}X_j}\rightarrow \sum_{j=1}^n{c_{t_ij}X_j}+\sum_{j=1}^k{\beta_{ij}Y_j} \rightarrow \sum_{j=1}^n{b_{t_ij}X_j}.
    $$
    
Suppose further that the new species $Y$ enter nondegenerately into $G'$ in the sense that the matrix $\beta :=(\beta_1|\beta_2|\dots|\beta_m)$, which has as columns the stoichiometric coefficientes of the complexes added to the modified reactions, has rank $m$ ($k\geq m$). Assume also that $a_i,b_i$ and $c_i$ are nonnegative vectors, where any or all may coincide. If $G$ admits nondegenerate multistationarity then so does $G'$. (See also \cite{Intermediates-Feliu_Wiuf}[Theorem 5.1])
\end{theorem}

\begin{theorem}[Adding a dependent reaction, \textit{Theorem 1} in \cite{Banaji2018}] \label{thm:depb} Let $G=(\X,\C,\R)$ be a network, and let $G'$ arise from $G$ by adding to $\mathcal{R}$ a new irreversible reaction with reaction vector $y$ which is a linear combination of reaction vectors of $\mathcal{R}$. If $G$ admits nondegenerate multistationarity, then so does $G'$.
    
\end{theorem}

This theorem allows us, in particular, to add to a network with the capacity for nondegenerate multistationarity the reverse of an existing irreversible reaction without affecting such capacity.

\subsection{Deficiency Theory}

Chemical reaction network theory, initiated by Horn, Jackson and Feinberg in the 1970s, aims at analyzing reaction networks independently of the choice of reaction rates. One key area of progress is deficiency theory. Deficiency theory provides a powerful tool to determine, among other things, if a network has the capacity for multistationarity or not just from its deficiency $\delta$, a quantity that is independent of choices of labeling and of the particular kinetics, as it is a purely structural invariant of the network. 

\begin{definition}
    Let $G = (\X, \C, \R)$ be a network with $\ell$ connected components (as a graph), and let $\imG$ be its stoichiometric (linear) subspace. The \textit{deficiency} of $G$ is given by the non-negative integer
    \[\delta = |\mathcal{C}| - \ell - \dim(\imG).\]
\end{definition}

Computationally, we count vertices and connected components from the reaction network seen as a graph and we compute $\dim(\imG)=\text{rank}(\Gamma)$, from a matrix of conservation laws $W$.

Deficiency also has a geometric interpretation. For this we first need to define \textit{affine} independence: We say that the vectors $y_0,\dots,y_N\in\mathbb{R}^M$ are affinely independent if $y_1 - y_0,\dots, y_N - y_0$ are linearly independent. Note that this is a property independent of the ordering of the vectors, and it can intuitively be thought of as the points $y_0,\dots,y_N\in\mathbb{R}^M$ being in general position.

We have the following equivalence for $G$: $\delta = 0$ if and only if the complex vectors within each connected component $G_i$ are affinely independent, and the stoichiometric subspaces $\imG_i$ associated to each $G_i$ are linearly independent \cite{CraciunJohnston2020,CraciunJinYu2023}.

Recall the definition of a reaction network being weakly reversible (\cref{def:weakly_reversible}). The following result, due to Horn and Feinberg, provides a way to ensure monostationarity of a network using its deficiency:

\begin{theorem}[Deficiency Zero Theorem \cite{Horn1972necessary,Feinberg1972complex,Feinberg1974dynamics}]\label{thm:DefZero}
    Let $G = (\X, \C, \R)$ be a network with deficiency $\delta=0$. Then, $G$ is monostationary. Moreover,
    \begin{enumerate}
        \item If $G$ is weakly reversible, then it admits a unique positive steady state (for each choice of  reaction rates) for each stoichiometric compatibility class, and this steady state is locally asymptotically stable.
        
        \item If $G$ is \textit{not} weakly reversible, then it does not admit any positive steady states (regardless of the choice of reaction rates).
    \end{enumerate}
\end{theorem}

 We will use this theorem to prove monostationarity in \Cref{sec:n-site_monost,sec:cascade_monost}.

\begin{remark}\label[remark]{rmk:monomolecular}
    Consider a monomolecular network, that is, a network where all complexes either have one species or they are the $0$ complex. By the geometric interpretation in terms of affine and linear independence, we can see that the deficiency of the network is 0. Thus, all monomolecular networks are monostationary (if they are weakly reversible, they have one steady state, otherwise they have none).
\end{remark}

\subsection{Sequential $n$-site double phosphorylation cycle}\label{sec:n-site}
Let us present the (sequential) $n$-site double phosphorylation cycle, which we will denote by $\mathcal{P}^n$. It is defined as

\begin{small}
 \begin{align}\label{eq:n-site}
S_0 + E \ce{<=>} ES_0 \ce{->}  \dots \ce{<=>}  & ES_{i} \ce{->} S_{i+1}+E  \ce{<=>} \dots \ce{<=>}  ES_{n-1}  \ce{->} S_n+E \nonumber\\
S_n + F  \ce{<=>} FS_n \ce{->} \dots \ce{<=>}  & FS_{i}  \ce{->} S_{i-1}+F \ce{<=>} \dots \ce{<=>}  FS_1 \quad  \ce{->} S_0+F.  \\
\nonumber
\end{align}
\end{small}

We will always assume that we are working with the sequential model in this paper.  It is sometimes convenient to highlight the role of different species in the network, and in this line we will classify species in $\mathcal{P}^n$ into three types. As $E,F$ act as catalysts in the inter-conversion between $S_i$ and $S_{i+1}$ ($0<i<n-1$), we call the species $E, F$ \textit{enzymes}, and the species $S_0, \dots, S_n$ \textit{substrates}. The species $ES_k$, $FS_{k+1}$ for $k=0,\dots,n-1$ are \textit{intermediate species}  because they take part in necessary intermediate steps of the phosphorylation/dephosphorylation processes, acting as an intermediary between different \textit{complexes} consisting of one \textit{enzyme} and one \textit{substrate}.

For a fixed $n>0$, $\mathcal{P}^n$ has $3n+3$ species and $6n$ reactions. If $x_X$ denotes the concentration of species $X$, then the network has the following conservation laws:
    \begin{align}\label{eq:cl_n-site}
        L_E: & \quad x_E + \sum_{i=0}^{n-1} x_{ES_i} = T_E\notag \\
        L_F: & \quad x_F + \sum_{i=1}^{n} x_{FS_i} = T_F\\
        L_S: & \quad \sum_{i=0}^n x_{S_i} + \sum_{i=0}^{n-1} x_{ES_i} + \sum_{i=1}^n x_{FS_i} = T_S\notag
    \end{align}
    where $T=(T_E, T_F, T_S)$ takes values in $\mathbb{R}^3_{>0}$.

Note that this network admits multistationarity for $n\geq 2$ (\cite{WS}, see also \cite{BIHAN2020367,CIK-2site}). The network is monostationary for $n=1$, as it satisfies the conditions of the deficiency one theorem \cite{feinberg-1987}.

In this work, we are interested in analyzing the capacity for multistationarity of semi-open versions of this network for $n\geq 2$. According to \cite{entrapped}, when opening all species the network always admits multistationarity (as the closed network admits multistationarity).

The results in \cite{cappelletti} allow us to know that the number of nondegenerate steady states is preserved in $\mathcal{P}^2$ when we add outflows to all species of this network, and inflows to certain subsets of them, under prescribed technical conditions of these subsets.

\medskip

We take the 2-site double phosphorylation cycle $\mathcal{P}^2$, and analyze $\mathcal{P}^2_{\E\rightleftharpoons 0}$ for different sets $\E$ of substrates and enzymes. The table below, generated using the Chemical Reaction Network Toolbox \cite{crn_toolbox}, shows some resulting networks which admit nondegenerate multistationarity (Multist.) and some which do not: 

\begin{table}[H]
\begin{tabular}{|c|c|c|}
\hline
Index & Species set $\E$ & Multist. for $\mathcal{P}^2_{\E\rightleftharpoons 0}$\\ \hline
1 & E / F          & Yes                                        \\ \hline
2 & E and F      & No                                     \\ \hline
3 & $S_0$ / $S_1$ / $S_2$         & Yes                                        \\ \hline
4 & $S_0$ and $S_1$/ $S_0$ and $S_2$/$S_1$ and $S_2$         & Yes                                        \\ \hline
5 & $S_0,S_1,S_2$     & Yes                                        \\ \hline
6 & E, F, $S_0$ / E, F, $S_1$ / E, F, $S_2$    & No                                     \\ \hline
7 & E, $S_0$        & Yes                                        \\ \hline
8 & E, $S_1$        & No                                     \\ \hline

\end{tabular}
\end{table}
\medskip
Even when intermediate species are excluded, varied behavior is observed for different combinations of enzymes and substrates being open. The table gathers only a few choices of $\E$, most of which our results can explain, but leaving out many which they cannot explain. More precisely, in this paper, we give theorems that encompass networks 1-6, proving this behavior of $\mathcal{P}^n$ for all $n\geq 2$. Networks 7 and 8 are included in the table only to show that the behavior for certain choices can be quite unpredictable at first sight.

\medskip

\section{Substrate‐Opening in the $n$‐Site Cycle}\label{sec:substrates}

\subsection{Overview and Main Statement}

As explained in \Cref{sec:n-site}, it is known that the $2$-site double phosphorylation cycle, which we will call $\mathcal{P}^{2}$ here, admits nondegenerate multistationarity. In this section, we shall prove that for any subset $\E$ of substrate species, the open network $\mathcal{P}^2_{\E\rightleftharpoons 0}$ also does and that this is also the case for $\mathcal{P}^{n}$, for any $n\geq 2$:

\begin{theorem}\label{thm:main_multist}
Consider $n\ge2$. The network $\mathcal{P}^{n}_{\mathcal{S}\rightleftharpoons0}$ obtained from $\mathcal{P}^{n}$ by opening an arbitrary subset of substrates $\mathcal{S}\subseteq \{S_0,\dots,S_n\}$ admits nondegenerate multistationarity.
In particular,
\[
\cappos\bigl(\mathcal{P}^{n}_{\mathcal{S}\rightleftharpoons0}\bigr)\geq 2.
\]
\end{theorem}

This result follows by induction on $n$, provided by the following ingredients: \Cref{prop:base_case} gives the base case for the induction, $n=2$. Given fixed $n\geq 2$ and $i\in \{ 1,\ldots ,n\}$, \Cref{thm:induction} states that if $\mathcal{P}^{n}_{S_i\rightleftharpoons0}$ admits nondegenerate multistationarity, then $\mathcal{P}^{n+1}_{S_i\rightleftharpoons0}$ also does. For $\mathcal{P}^{n+1}_{S_{n+1}\rightleftharpoons 0}$, we will need \Cref{rem:symmetry}, which shows that, by symmerty, if $\mathcal{P}^{n}_{S_{i}\rightleftharpoons 0}$ has nondegenerate multistationarity, so does $\mathcal{P}^{n}_{S_{n-i}\rightleftharpoons 0}$. The extension to opening any subset $\mathcal{S}$ follows from the fact that, once a substrate is open, opening any other substrate does not modify the stoichiometric subspace, together with \Cref{thm:enlarging_preserving_stoichiometry}.  

\Cref{sec:thm_multist} is devoted to \Cref{thm:induction} and the technical lemmas needed for its proof. In the rest of this section we present the remaining results mentioned in the previous paragraph.  The complete proof of \Cref{thm:main_multist} is given in \Cref{sec:proof_main_substrates}. 

 The next remark allows us to explain the dynamical behavior of $\mathcal{P}^{n}_{S_{n+1}\rightleftharpoons 0}$ via that of $\mathcal{P}^{n}_{S_{0}\rightleftharpoons 0}$. This result will also be useful to extend the inductive step to $\mathcal{P}^{n+1}_{S_{n+1}\rightleftharpoons 0}$. 

\begin{remark}\label[remark]{rem:symmetry} 
  $\mathcal{P}^{n}_{S_{i}\rightleftharpoons 0}$ is dynamically equivalent to $\mathcal{P}^{n}_{S_{n-i}\rightleftharpoons 0}$, that is, they \textit{both have the same right-hand side of the mass action differential equations with a swap in labels} \cite{johnston-2011}. In particular, from $\mathcal{P}^{n}_{S_{i}\rightleftharpoons 0}$ one can rename species as follows: exchange substrate indices to denote the number of unphosphorylated sites (instead of the number of phosphorylated sites), so that $S_i$ becomes $S_{n-i}$, then exchange labels of $E$ and $F$, and finally exchange the names of intermediate species accordingly ($ES_i$ becomes $FS_{n-i}$). The resulting network looks like\\
\begin{small}
 \begin{align}\label{eq:n-site_op_r}
S_n + F  \ce{<=>} FS_n \ce{->} \dots \ce{<=>}  & FS_{i}  \ce{->} S_{i-1}+F \ce{<=>} \dots \ce{<=>}  FS_1 \quad  \ce{->} S_0+F \nonumber  \\
S_0 + E \ce{<=>} ES_0 \ce{->}  \dots \ce{<=>}  & ES_{i} \ce{->} S_{i+1}+E  \ce{<=>} \dots \ce{<=>}  ES_{n-1}  \ce{->} S_n+E \\
S_{n-i} \ce{<=>}0, \nonumber
\end{align}
\end{small}
which is exactly the network $\mathcal{P}^{n}_{S_{n-i}\rightleftharpoons 0}$

\end{remark}

Let us focus now on $\mathcal{P}^2$:

\begin{proposition}\label[proposition]{prop:base_case}The $2$-site double phosphorylation cycle with any of the substrates open, $\mathcal{P}^2_{S\rightleftharpoons 0}$, $S\in\{S_0,S_1,S_2\}$ admits nondegenerate multistationarity.    
\end{proposition}

\begin{proof}
    The proof for $\mathcal{P}^{2}_{S_{1}\rightleftharpoons 0}$ and $\mathcal{P}^{2}_{S_{0}\rightleftharpoons 0}$ admitting nondegenerate multistationarity is given in \Cref{sec:induction_base}. To extend this for $\mathcal{P}^{2}_{S_{2}\rightleftharpoons 0}$, note that $\mathcal{P}^{n}_{S_{i}\rightleftharpoons 0}$ is \textit{dynamically equivalent} to $\mathcal{P}^{n}_{S_{n-i}\rightleftharpoons 0}$ by symmetry (\Cref{rem:symmetry}).
\end{proof}

Once one of the substrates of $\mathcal{P}^2$ is open the new network has only two conservation laws, namely $L_E$ and $L_F$ from \eqref{eq:cl_n-site}. Now opening any of the other substrates does not alter these conservation laws and therefore does not change the stoichiometric subspace. This is also the case for any $n\geq 2$. As a consequence, the following result shows that by opening any other substrate the capacity for nondegenerate multistationarity is preserved:

\begin{theorem}\label{thm:enlarging_preserving_stoichiometry}\cite[Theorem 3.1]{joshi-shiu-II}    Let $G$ be a network, and let $\widetilde{G}$ be a subnetwork of $G$ such that $G$ and $\widetilde{G}$ share the same stoichiometric subspace. If $\widetilde{G}$ admits nondegenerate multistationarity with mass action kinetics, so does $G$. Moreover, $G$ admits at least as many nondegenerate positive steady states in some stoichiometric compatibility class as the subnetwork $\widetilde{G}$ does.
\end{theorem}

\subsection{Inductive Construction of Steady States}\label{sec:thm_multist}
In what follows, we will denote by $x_X$ the concentration of species $X$, and by $f_X$ the corresponding species formation function. 

The following technical Lemmas will be necessary for the proof of  \Cref{thm:induction}. The first of them allows us to find a positive steady state for $\mathcal{P}^{n+1}_{S_i\rightleftharpoons 0}$ from one of $\mathcal{P}^{n}_{S_i\rightleftharpoons 0}$ for $i\in\{0,\ldots  ,n\}$, while the second checks that the nondegeneracy of the first guarantees that of the latter.

\begin{lemma}\label[lemma]{lemma:n+1_steadystate}
    Let $i\in\{0,\ldots ,n\}$, and let $G=\mathcal{P}^n_{S_i\rightleftharpoons 0}$ be the $n$-site double phosphorylation cycle, with set of species $\X$, as in \eqref{eq:n-site}, with the two additional reactions $S_i \rightleftharpoons 0$. Let $\bar{G}$ be the  result from adding to $G$ the reactions $S_n +E\rightarrow S_{n+1}+E$ and $S_{n+1} +F\rightarrow S_{n}+F$,  that is, the network

 \begin{align}\label{eq:n-site_ext1}
S_0 + E \ce{<=>} ES_0 \ce{->}  \dots & \ce{<=>}  ES_{n-1} \ce{->} S_n+E \ce{->} S_{n+1}+E\notag \\
S_{n+1}+F \ce{->} S_n + F  \ce{<=>} FS_n \ce{->} \dots & \ce{<=>}  FS_1 \quad \ce{->} S_0+F\notag\\
S_i \ce{<=>}0,
\end{align}
\\\
    If $x'$ is a positive steady state of the mass action system $(G,\k^*)$, then $\bar{x}'$ defined as
\begin{align}\label{eq:new_ss}
    \bar{x}'_{S_{n+1}}&=\frac{x'_{S_n}x'_E}{x'_F}, \\
    \bar{x}'_X&=x'_X  \quad \text{ for any $X\in\mathcal{X}$,}\notag
\end{align}
is a positive steady state of $(\bar{G},\bar{\k}^*)$, where $\bar{\k}^*_{y\rightarrow y'}=\k^*_{y\rightarrow y'}$ for any reaction which was already in $G$ and any $a\in\mathbb{R}_{>0}$ can be chosen as the reaction rate in
 \begin{align*}
 S_n+E  \ce{->[a]} S_{n+1}+E \\
S_{n+1}+F \ce{->[a]} S_n+F .
\end{align*}

Moreover, if $x'$ lies in the stoichiometric class with total amounts $T=(T_E,T_F)$, then $\bar{x}'$ lies in the stoichiometric class of $G'$ with total amounts $\bar{T}=(T_E,T_F)$. 
\end{lemma}

\begin{proof}
Let $\frac{dx}{dt}=f(\k,x)$ be the mass action ODE system for $G$ and let $\frac{d\bar{x}}{dt}=g(\bar{\k},\bar{x})$ be the mass action ODE system for $\bar{G}$.
Due to the nature and the reaction rates of the reactions added, we get:
\begin{align*}
    g_{S_{n+1}}(\bar{\k}^*,\bar{x})= &a\bar{x}_{S_n}\bar{x}_E-a\bar{x}_{S_{n+1}}\bar{x}_F\\
    g_{S_n}(\bar{\k}^*,\bar{x})= &f_{S_n}(\k^*,\bar{x})-a\bar{x}_{S_n}\bar{x}_E+a\bar{x}_{S_{n+1}}\bar{x}_F\\
    g_X(\bar{\k}^*,\bar{x})= &f_X(\k^*,\bar{x}) \text{    (for any $X \in \mathcal{X}\setminus\{S_n\}$)}
\end{align*}

Now the first part of the statement is straightforward after noting that for this $\bar{x}'$ it holds that $f(\k^*,\bar{x}')=0$ trivially, and then substituting the values of ${x}'_{S_{n+1}}$ makes the additional terms vanish, yielding $g(\bar{\k}^*,\bar{x}')=0$.

The conservation laws do not change,
\begin{align*}
    \bar{T_E}=& \sum_{i=0}^{n-1}\bar{x}'_{ES_i}+\bar{x}'_E =\sum_{i=0}^{n-1}{x'}_{ES_i}+{x'}_E =T_E\\
    \bar{T_F}=& \sum_{i=1}^{n}\bar{x}'_{FS_i}+\bar{x}'_F=\sum_{i=1}^{n}{x'}_{FS_i}+{x'}_F=T_F,
\end{align*}
and in particular the conservation laws of $\bar{G}$ do not involve the new species, so the total amounts for $\bar{x}'$ are the same as for $x'$.
\end{proof}

\begin{lemma}\label[lemma]{lemma:n+1_nondeg}
   Let $i\in\{0,\ldots ,n\}$, and let $G=\mathcal{P}^n_{S_i\rightleftharpoons 0}$ and $\bar{G}$ be as in \eqref{eq:n-site_ext1}. 
   If $x'$ is a nondegenerate positive steady state of $(G,\k^*)$, then $\bar{x}'$ defined as in \eqref{eq:new_ss} is a nondegenerate positive steady state of $(\bar{G},\bar{\k}^*)$, where $\bar{\k}^*$ is as in \Cref{lemma:n+1_steadystate}.
\end{lemma}

\begin{proof}If $x'$ is nondegenerate for $(G,\k^*)$, it means the right kernel of the Jacobian matrix of $f(\k^*,x)$ at $x'$ has no intersection with the stoichiometric subspace of $G$. We will show that, if this is the case,  then the right kernel of the Jacobian matrix of $g(\bar{\k}^*,\bar{x})$ at $\bar{x}'$ also has no intersection with the stoichiometric subspace of $\bar{G}$, and therefore $\bar{x}'$ is nondegenerate for $(\bar{G},\bar{\k}^*)$.\\
Let $J(x',\kappa^*)$ be the Jacobian of $(G,\k^*)$ at $x'$, with entries 
\begin{equation}
J(x',\kappa^*)_{ij}=\frac{\partial f_{X^{G}_i}}{\partial x_{X^{G}_j}}\Big|_{x'} \nonumber    
\end{equation}
for some ordering on the species of $G$, $\mathcal{X}_{G}=\{X^{G}_1,X^{G}_2\dots X^{G}_{3n+3}\}$. Without loss of generality, we may assume that $X^{G}_1=S_n$, $X^{G}_2=E$ and $X^{G}_3=F$. Consider now the Jacobian matrix for $(\bar{G},\bar{\k}^*)$ at $\bar{x}'$, $\bar{J}(\bar{x}',\bar{\k} ^*)$, with entries
\begin{equation}
\bar{J}(\bar{x}',\bar{\kappa}^*)_{ij}=\frac{\partial g_{X^{\bar{G}}_i}}{\partial x_{X^{\bar{G}}_j}}\Big|_{\bar{x}'} \nonumber    
\end{equation}
for some ordering on the species of $\bar{G}$, $\mathcal{X}_{\bar{G}}=\{X^{\bar{G}}_1,X^{\bar{G}}_2\dots X^{\bar{G}}_{3n+4}\}$. Let the ordering be such that $X^{\bar{G}}_i=X^{G}_i$ for $i\leq3n+3$, and $X^{\bar{G}}_{3n+4}=S_{n+1}$.\\
By the structure of $g$ and it's dependence on $f$, we can write $\bar{J}$ as\\
$$
\bar{J}=\left[ 
\begin{array}{c|c} 
  J_{(3n+3)\times(3n+3)}+A_{(3n+3)\times(3n+3)} & B_{(3n+3)\times1} \\ 
  \hline 
  C_{1\times(3n+3)} & D_{1\times1} 
\end{array} 
\right] ,
$$
with $J$ as defined before for $G$, and
$$
A=\left[ 
\begin{array}{cccccc|c} 
 -a\bar{x}_E & -a\bar{x}_{S_n} & a\bar{x}_{S_{n+1}} & 0 & \cdots & 0& 0 \\
 \hline
 0 &  &\cdots &  && 0 & 0 \\ 
  \vdots & &\ddots& & &\vdots& \vdots\\
 0 & & \cdots & & & 0 & 0 \\
\end{array} 
\right],
$$
$$
C=\left[ 
\begin{array}{cccccc} 
  a\bar{x}_E & a\bar{x}_{S_n} & -a\bar{x}_{S_{n+1}} & 0  &\cdots & 0
 \end{array} 
\right],
$$
$$
B=\left[ 
\begin{array}{c} 
 a\bar{x}_F   \\
 0     \\
 \vdots\\
 0
\end{array} 
\right],
$$

$$
D=\left[ 
\begin{array}{c} 
-a\bar{x}_F
 \end{array} 
\right].
$$

\noindent Let $\textbf{c}=(c_1,\ldots ,c_{3n+4})\in \ker (\bar{J})$. 
Using the last row of $\bar{J}$, we know that

\begin{equation}\label{eq:first_row}
    a\bar{x}_E c_1+a\bar{x}_{S_n} c_2-a\bar{x}_{S_{n+1}}c_3-a\bar{x}_Fc_{3n+4}=0 .
\end{equation}
Consider the vector $\textbf{c}'$ obtained by removing the last entry of $\textbf{c}$. 
Let $J_{j.}$ denote the $j^{th}$ row of $J$. Now $J\textbf{c}'$ gives
\begin{eqnarray}\label{eq:last_row}
    J_{1.}\textbf{c}'-a\bar{x}_E c_1-a\bar{x}_{S_n} c_2+a\bar{x}_{S_{n+1}}c_3+a\bar{x}_Fc_{3n+4}=0  
\end{eqnarray}
as first row, and $J_{j.}\textbf{c}'=0$ for all other rows. 
But \eqref{eq:first_row} and \eqref{eq:last_row} together give $J\textbf{c}'=0$. Thus $\textbf{c}'$ is a vector in the right kernel of $J(x',\kappa^*)$ whenever \textbf{c} is a vector in the right kernel of $\bar{J}(\bar{x}',\bar{\kappa}^*)$.\\
\medskip
It is left to note that the conditions on a vector $\textbf{c}'$ for being in the stoichiometric subspace of $G$ are
\begin{align*}
    \textbf{c}_E+\sum_{i=0}^{n-1}{\textbf{c}_{ES_i}}=0&\\
    \textbf{c}_F+\sum_{i=1}^{n}{\textbf{c}_{FS_i}}=0& ,
\end{align*}
but the same conditions apply to $\textbf{c}$ for being in the stoichiometric subspace of $\bar{G}$. Now, if $x'$ is nondegenerate, then no $\textbf{c}'\in \ker(J(x',\kappa^*))$ satisfies these conditions, nor does any vector $\textbf{c}\in \ker(\bar{J}(\bar{x}',\bar{\kappa}^*))$, so $\bar{x}'$ is nondegenerate too.
\end{proof}

We can proceed now to give the result that provides the inductive step for \Cref{thm:main_multist}.

\begin{theorem}\label{thm:induction}
 Let $i\in\{0,\ldots ,n\}$. If $\mathcal{P}^n_{S_i\rightleftharpoons 0}$ admits nondegenerate multistationarity with mass action kinetics, then $\mathcal{P}^{n+1}_{S_i\rightleftharpoons 0}$ does too.
\end{theorem}

\begin{proof}
Given two distinct nondegenerate steady states in one stoichiometric class, $x'$ and $x''$, of $G=\mathcal{P}^n_{S_i\rightleftharpoons 0}$, we get two distinct nondegenerate steady states of $\bar{G}$, $\bar{x}'$ and $\bar{x}''$ as in \Cref{eq:n-site_ext1}, from \Cref{lemma:n+1_steadystate,lemma:n+1_nondeg}, which also belong to one stoichiometric class.\\

Now we use \cref{thm:interb} (\cite[Theorem 6]{Banaji2018}), to add new intermediate species in the form of unique monomolecular complexes (a different one to each of a subset of reactions) to a reaction network which admits nondegenerate multistationarity (which is called multiple positive nondegenerate equilibria, or MPNE in \cite{Banaji2018}), knowing that the new network also admits nondegenerate multistationarity. For $\bar{G}$, this process on adding two intermediates $ES_n$ and $FS_{n+1}$ results in the network

\begin{footnotesize}\label{eq:n-site_ext2}
 \begin{align}
S_0 + E \ce{<=>} &ES_0 \ce{->}  \dots\ce{<=>}  ES_{n-1} \ce{->} S_n+E \ce{->} ES_n \ce{->} S_{n+1}+E\notag \\
S_{n+1}+F \ce{->} FS_{n+1}\ce{->} S_n + F  \ce{<=>} &FS_n \ce{->} \dots \ce{<=>}  FS_1 \quad \ce{->} S_0+F\notag\\
S_i \ce{<=>}0.
\end{align}
\end{footnotesize}
Also \Cref{thm:depb} (\cite[Theorem 1]{Banaji2018}) allows us to add the reverse of an existing reaction preserving the capacity for nondegenerate multistationarity (see \cite[Remark 4.1]{Banaji2018}). Hence  

\begin{footnotesize}\label{eq:n+1-site}
 \begin{align}
S_0 + E \ce{<=>} &ES_0 \ce{->}  \dots\ce{<=>}  ES_{n-1} \ce{->} S_n+E \ce{<=>} ES_n \ce{->} S_{n+1}+E\notag \\
S_{n+1}+F \ce{<=>} FS_{n+1}\ce{->} S_n + F  \ce{<=>} &FS_n \ce{->} \dots \ce{<=>}  FS_1 \quad \ce{->} S_0+F\notag\\
S_i \ce{<=>}0,
\end{align}
\end{footnotesize}
which is just $\mathcal{P}^{n+1}_{S_i\rightleftharpoons 0}$, also admits nondegenerate multistationarity.\\

We can also obtain $\mathcal{P}^{n+1}_{S_i\rightleftharpoons 0}$ as an \textit{extension model} of $\bar{G}$, as defined in \ref{extension-model}. Since $\bar{G}$ has at least 2 distinct non-degenerate steady states, $\mathcal{P}^{n+1}_{S_i\rightleftharpoons 0}$ also has at least 2 distinct non-degenerate steady states using \ref{extension-model-multist}.

\end{proof}

\subsection{Numerical illustration}\label{sec:induction_base}
We get the following numerical results using the Chemical Reaction Network Toolbox \cite{crn_toolbox}.\\

Consider $\mathcal{P}^{2}_{S_0\rightleftharpoons 0}$ with reaction rates as follows\\

 \begin{small}
 \begin{align}\notag
 \begin{split}
S_0 + E \ce{<=>[3.436][1.718]} ES_0 \ce{->[1.718]} S_1+E \ce{<=>[2.971][0.316]} ES_1 \ce{->[0.316]} S_2+E \\
S_2 + F  \ce{<=>[37.471][0.316]} FS_2 \ce{->[0.316]} S_{1}+F \ce{<=>[33.005][1.718]} FS_1 \ce{->[1.718]} S_{0}+F \\
S_0 \ce{<=>[1][1]} 0.\\
\end{split}
\end{align}
\end{small}
Let $x=(x_{S_0},x_{S_1},x_{S_2},x_{ES_0},x_{ES_1},x_{FS_1},x_{FS_2},x_{E},x_{F})$. This system is multistationary with two steady states. Indeed, $x^1=(1,1.156,1.018,0.581,3.163,0.581,3.163,0.581,0.052)$ and
$x^2=(1,0.156,0.018,1.581,1.163,1.581,1.163,1.581,1.052)$ are two nonzero, nondegenerate steady states in the same stoichiometric class ($T_E=4.325,T_F=3.796$). By symmetry (see \Cref{rem:symmetry}), the 2-site phosphorylation cycle with $S_2$ open also admits nondegenerate multistationary.\\

\medskip

Consider now $\mathcal{P}^{2}_{S_1\rightleftharpoons 0}$ with reaction rates as follows\\
 \begin{small}
 \begin{align}\notag
 \begin{split}
S_0 + E \ce{<=>[68.609][7.297]} ES_0 \ce{->[7.297]} S_1+E \ce{<=>[79.347][39.673]} ES_1 \ce{->[39.673]} S_2+E \\
S_2 + F  \ce{<=>[186.499][19.836]} FS_2 \ce{->[19.836]} S_{1}+F \ce{<=>[29.19][14.595]} FS_1 \ce{->[14.595]} S_{0}+F \\
S_1 \ce{<=>[1][1]} 0.\\
\end{split}
\end{align}
\end{small}
This system has two nondegenerate steady states in the same stoichiometric class too ($T_E=0.187,T_F=0.187$), given by $x^1=(1.156,1,0.156,0.137,0.025,0.068,0.05,0.025,0.068)$ and\\
$x^2=(0.156,1,1.156,0.05,0.068,0.025,0.137,0.068,0.025)$.

\medskip

\subsection{Proof of non-degenerate multistationarity with substrates open}\label{sec:proof_main_substrates}

With the base case in \cref{sec:induction_base}, we can prove \Cref{thm:main_multist}, 

\begin{proof} (of \Cref{thm:main_multist})
     Assume that $\mathcal{S}$ is non-empty (otherwise the statement is trivial) and pick $i$ such that $S_i\in\mathcal{S}$. We start from the knowledge that $\mathcal{P}^{2}_{S_1\rightleftharpoons 0}$ and $\mathcal{P}^{2}_{S_2\rightleftharpoons 0}$ admit nondegenerate multistationarity (\ref{prop:base_case}), and using \Cref{thm:induction} and \Cref{rem:symmetry}, possibly repeatedly, we can ensure that $\mathcal{P}^{n}_{S_i\rightleftharpoons 0}$ also admits nondegenerate multistationarity. Now, opening the substrates $S_{j(\neq i)}\in\mathcal{S}$ does not modify the stoichiometric subspace, hence $\mathcal{P}^{n}_{\mathcal{S}\rightleftharpoons 0}$ admits nondegenerate multistationarity. 
 \end{proof}
 
    As a demonstration, for example, to prove that $\mathcal{P}^{5}_{S_3\rightleftharpoons 0}$  admits nondegenerate multistationarity:\\
    Start with $\mathcal{P}^{2}_{S_1\rightleftharpoons 0}$, which admits nondegenerate multistationarity\\
    By \Cref{thm:induction}, $\mathcal{P}^{3}_{S_1\rightleftharpoons 0}$ and further $\mathcal{P}^{4}_{S_1\rightleftharpoons 0}$ admits nondegenerate multistationarity\\
    By \Cref{rem:symmetry}, $\mathcal{P}^{4}_{S_1\rightleftharpoons 0}$ is the same as $\mathcal{P}^{4}_{S_3\rightleftharpoons 0}$,\\
    By \Cref{thm:induction} again, $\mathcal{P}^{5}_{S_3\rightleftharpoons 0}$ admits nondegenerate multistationarity.\\

This implies that for any length of the double phosphorylation cycle, the substrates traveling in and out of the system do not disrupt the multistationary behavior of the cycle. They are free to participate in other reactions, which is important considering the substrates are often phosphorylated for use in different processes, for example, in cascades (\cite{hell-2016}), 
 and the system is robust against these exchanges.

\section{Enzyme‐Opening and monostationarity}\label{sec:enzymes}

In contrast to what happens with the $n$-site double phosphorylation cycle when opening a set of substrates, opening a set of enzymes makes it lose its capacity for multistationarity. The strategy followed to prove this is based on the network modifications defined in \Cref{sec:network_modifications}.  The same strategy is useful with other networks, such as phosphorylation cascades, like the MAPK cascade (see \Cref{sec:cascade_monost}). Along this section we present the tools that allow us to prove this for the two networks mentioned, and which can also be applied to further ones with similar structure. \Cref{sec:overview_monost} presents the main result and \Cref{sec:ACR} the tools to prove it. \Cref{sec:n-site_monost} and \Cref{sec:cascade_monost} are devoted to the application of these to the $n$-site double phosphorylation network and phosphorylation cascades respectively.

\subsection{Overview and main statement}\label{sec:overview_monost}

\begin{theorem}\label{thm:enzyme-open-mono}
  Let $G=(\X,\C,\R)$ be a network and let $\E\subset \X$ be a set of independently conserved species. Then
  \[
  \cappos\bigl(G_{\E\rightleftharpoons0}\bigr)
  \leq \cappos(G_{-\E}).
  \]
  Equivalently, any upper bound on the number of positive steady states of $G_{-\E}$ that holds for every choice of positive rate constants and every stoichiometric compatibility class also holds for $G_{\E\rightleftharpoons0}$. In particular, if $G_{-\E}$ is monostationary, then $G_{\E\rightleftharpoons 0}$ is monostationary.
  
\end{theorem}

The first step in the proof requires noting that when opening a set of independently conserved species (see \Cref{def:independently_conserved}) in a network, the resulting network has \textit{Absolute Concentration Robustness (ACR)} in all those species. That is, for each of these species there is a fixed value at positive steady state (see \cite{S-F} ans \Cref{sec:ACR}). This will be the content of \Cref{thm:open-is-ACR}.

The second step involves projecting $G_{\E\rightleftharpoons 0}$ over $\X\setminus\E$, and noting that the number of  possible  positive steady states of the projected network  in a compatibility class  is at least that of $G_{\E\rightleftharpoons 0}$. Therefore, if $G_{\E\rightleftharpoons 0}$ admits multistationarity, so does $(G_{\E\rightleftharpoons 0})_{-\E}=G_{-\E}$ or, equivalently, if the later is monostationary so is $G_{\E\rightleftharpoons 0}$. 

As a consequence, detecting monostationarity for a network with less species will allow us to ensure monostationarity for a bigger network, as long as the species removed are independently conserved. This will give us a way to, for example, prove monostationarity of the semi-open extension of the $n$-site phosphorylation cycle where the open species are the two enzymes (\Cref{sec:n-site_monost}). The same will be possible for phosphorylation cascades, like the MAPK cascade (\Cref{sec:cascade_monost}).

\begin{remark}
    \Cref{thm:enzyme-open-mono} holds for any network $G$ and any subset of species $\E$ such that $G_{-\E}$ is monomolecular. These networks yield a simple scenario for the application of 
    the theorem, as we will see in Section \ref{sec:n-site_monost}.
\end{remark}

\subsection{ACR, projection and inheritance}\label{sec:ACR}

We will say that a reaction network has \textit{Absolute Concentration Robustness (ACR)} in a certain species $X\in \X$ for a certain choice of kinetics, in the sense of \cite{S-F}, if the concentration of $X$ at positive steady state is constant with this choice of kinetics, independently of the stoichiometric compatibility class, provided that there is some positive steady state. We will be considering our networks with mass action kinetics, as we have done up until now. Still, this property or the particular constant value may depend on the specific reaction rates. We will say that the network has ACR if it does for all reaction rates for which there is some positive steady state. 

The property of ACR confers a network a strong robustness against perturbations, as the concentration of the species under consideration will always return to the same stationary value. Moreover, this value does not depend on the initial concentrations.

In this work, we will use the property of ACR as a tool to guarantee monostationarity. The following Lemma will allow us to establish a connection between the possible steady state concentrations of species that are independently conserved in a network via a subnetwork.
\begin{lemma}[Steady‐State Projection]\label[lemma]{lemma:SumSteadyStates}
    Let $G=(\X,\C,\R)$ be a network and let $\E\subset \X$ be a set of independently conserved species. Let $H=(\E,\C',\R')$ be any other network with set of species $\E$. If $z\in\mathbb{R}_{>0}^n$ is a steady state of $G\cup H$, then its projection $z_{\E}$ (\cref{def:projected_network}) is a steady state of $H$,  and $z$ is a steady state of $G$ as well. Conversely, if $z$ is a steady state of $G$ and its projection $z_{\E}$ is a steady state of $H$, then $z$ is a steady state of $G\cup H$ as well. 
\end{lemma}

Note that, in the lemma, $H$ does not need to be a subnetwork of $G$, and will not usually be one.

\begin{proof}
  Let us assume, without loss of generality, that $\E=\{X_i\}_{i=1,\ldots ,k}$ and write $\dfrac{dx}{dt} = f^G(x)$ and $\dfrac{d\tilde{x}}{dt} = f^H(\tilde{x})$
for the mass action ODE systems for $G$ and $H$ respectively, where $x=(x_1,\ldots ,x_n)$ and $\tilde{x}=(x_1,\ldots ,x_k)$. Now $G\cup H$ has $\X$ as set of species, and the mass action system $\dfrac{dx}{dt} = f^{G\cup H}(x)$ for this network is given by
        \begin{align*}
       \dfrac{dx_i}{dt} = & f^G_i(x_1,\dots,x_n) + f^H_i(x_1,\dots,x_k) \;\mbox{for}\, i=1,\ldots,k\\
       \dfrac{dx_i}{dt} = & f^G_i(x_1,\dots,x_n) \;\mbox{for}\, i=k+1,\ldots,n.
    \end{align*}

    By hypothesis, there exist $w_{ij}\in \mathbb{R}$ for $i=1,\dots,k$, $j=1,\ldots ,n$ that give a linear dependency relation among the polynomials $f^G_j$ for each $j=1,\ldots ,n$. Moreover, the $w_{ij}$ can be chosen in a way so that the $k\times k$-matrix $W_k=(w_{ij})_{i,j\in\{1,\ldots ,k\}}$ is invertible (see \Cref{rem:indep_conserved-invertible}):
    \begin{equation}
        \sum_{j=1}^n w_{ij} f^G_j(x_1,\dots,x_n) = 0,\quad i=1,\ldots ,k.
    \end{equation}
    Let $z = (z_1,\dots,z_n)$ be a positive steady state of $G \cup H$, then for $i=1,\dots,k$,
    \begin{align}
        0 &= \sum_{j=1}^n w_{ij} f^{G \cup H}_j(z_1,\dots,z_n) \notag \\
        &= \sum_{j=1}^n w_{ij} f^G_j(z_1,\dots,z_n) + \sum_{j=1}^k w_{ij} f^H_j(z_1,\dots,z_k) \label{eq:f_GcupH}\\
        &= \sum_{j=1}^k w_{ij} f^H_j(z_1,\dots,z_k). \notag
    \end{align}
    But, as $W_k$ is invertible, this implies that
    \begin{equation}
        f^H(z_1,\dots,z_k) = 0,
    \end{equation}
    so $(z_1,\dots,z_k)$ is a steady state of $H$.
    
    Note that then
    \[
    0 = f^{G \cup H}(z_1,\dots,z_n) = f^G(z_1,\dots,z_n) + f^H(z_1,\dots,z_k) = f^G(z_1,\dots,z_n)
    \]
    proving that $z$ is also a steady state of $G$. The converse follows from the second equality in the previous equation.
    
\end{proof}

Note that if $G$ or $H$ has no positive steady states, then neither does $G\cup H$. Let us recall that ACR requires the existence of at least one positive steady state. The following theorem is a direct consequence of the \Cref{lemma:SumSteadyStates}. Note that it slightly expands \cite[Theorem 27]{PhD-Torres}.

\begin{theorem}[ACR Emergence under Opening]\label{thm:open-is-ACR}  
Let $G=(\X,\C,\R)$ be a network and let $\E\subset \X$ be a set of independently conserved species.  Then either $G$ has no positive steady states and neither does $G_{\E\rightleftharpoons 0}$ or $G$ has a positive steady state and $G_{\E\rightleftharpoons 0}$ has ACR in all species of $\E$ with values $x_{E}=\frac{\kappa_{0\to E}}{\kappa_{E\to 0}}$ for $X \in \E$. 

More generally, if we open partially the species of $\E$, then:
\begin{itemize}
    \item if for any $X \in \E$ we have $0\ce{->} X$ but no $X\ce{->} 0$, then there are no steady states for the resulting network,
    \item if for any $X$ we have $X\ce{->} 0$ but no $0\ce{->} X$ then there can only be boundary steady states (with $x_{X}=0$).
\end{itemize}
\end{theorem}

\begin{proof}
    
    Consider $H=\{X\ce{<=>} 0:X\in\E\}$ so that $G_{\E\rightleftharpoons 0}=G\cup H$. If $G$ has no positive steady states, then by \Cref{lemma:SumSteadyStates} $G\cup H$ also does not have any positive steady states.
    
    Assume then that $G$ has a positive steady state $z \in \mathbb{R}^n_{>0}$. For $X\in \E$, choose $\kappa_{0\to X}$ and $\kappa_{X \to 0}$ such that $z_X = \kappa_{0\to X}/\kappa_{X\to 0}$. It is easy to check that $H$ has a unique steady state $z^*$, with $z^*_X = \kappa_{0\to X}/\kappa_{X\to 0}$ for each $X\in \E$. Thus, by \Cref{lemma:SumSteadyStates}, as $z_{\E} = z^*_{\E}$ is a steady state of $H$, then $z$ is a steady state of $G_{\E\rightleftharpoons 0}$. Furthermore, for any steady state $x\in \mathbb{R}^n_{>0}$ of $G_{\E\rightleftharpoons 0}$ we have that $x_{X} = \kappa_{0\to X}/\kappa_{X\to 0}$ for each $X\in \E$, so $G_{\E\rightleftharpoons 0}$ has ACR in all species of $\E$.

    Consider now a subnetwork $H'\subset H$. If $0\ce{->}X \in H'$ but $X \ce{->} 0\notin H'$ for any species $X\in \E$, then $H'$ has no steady states, and thus by \Cref{lemma:SumSteadyStates}, $G\cup H'$ has no steady states. Similarly, if $X \ce{->} 0\in H'$ but $0\ce{->} X \notin H'$, then $H'$ can only have boundary steady states with $x_{X}=0$, thus $G\cup H'$ can only have boundary states, where $z_{X}=0$.
\end{proof}

\begin{example}\label[example]{ex:open-is-ACR}
    Consider $G$ given by \eqref{eq:E-graph} as in \cref{ex:E-graph}. Note that it has a conservation law given by the conservation of mass, that is, the total amount of $X$ plus $Y$ plus $Z$ is conserved. Thus, the set $\E = \{Z\}$ is independently conserved, and by \cref{thm:open-is-ACR}, $G_{Z\rightleftharpoons 0}$ has ACR in $Z$, with value $x_Z = \kappa_{0\to Z}/\kappa_{Z\to 0}$. Note also that $G_{Z\rightleftharpoons 0}$ doesn't have further (linearly independent) any conservation laws.

    \begin{align}\label{tikz:open}
    \begin{aligned}
    \begin{tikzpicture}[scale=1.2]

\begin{axis}[view={105}{30}, axis line style=white, width=5.5cm, height=5.5cm, ticks=none, xmin=0, xmax=1.9, ymin=0, ymax=1.9, zmin=0, zmax=2.5]

    \tikzset{bullet/.style={inner sep=1pt,outer sep=1.5pt,draw,fill,Egraphcolor,circle}};
    \tikzset{bullet2/.style={inner sep=1pt,outer sep=1.5pt,draw,fill,orange,circle}};
    \tikzset{myarrow/.style={arrows={-stealth},thick,Egraphcolor}};
    \tikzset{myarrow2/.style={arrows={-stealth},thick,orange}};
    \tikzset{mygrid/.style={very thin,gray!50}};
    \tikzset{myaxis/.style={thick,gray}};    
    \newcommand{\ab}{2.5};    
 
    \draw[mygrid] (1,0,\ab) -- (1,0,0) -- (1,\ab,0);
    \draw[mygrid] (2,0,\ab) -- (2,0,0) -- (2,\ab,0);
    \draw[mygrid] (0,1,\ab) -- (0,1,0) -- (\ab,1,0);
    \draw[mygrid] (0,2,\ab) -- (0,2,0) -- (\ab,2,0);
    \draw[mygrid] (0,\ab,1) -- (0,0,1) -- (\ab,0,1);
    \draw[mygrid] (0,\ab,2) -- (0,0,2) -- (\ab,0,2);
        
    \draw[myaxis] (0,0,0) -- (\ab,0,0);    
    \draw[myaxis] (0,0,0) -- (0,\ab,0);    
    \draw[myaxis] (0,0,0) -- (0,0,\ab);    
    
    \node[bullet] (X) at (1,0,0) {};
    \node[bullet] (Y) at (0,1,0) {};
    \node[bullet] (Z) at (0,0,1) {};
    \node[bullet] (2Z) at (0,0,2) {};
    \node[bullet] (Y+Z) at (0,1,1) {};
    \node[bullet2] (0) at (0,0,0) {};
    
    \node [below right] at (X) {\small $\mathsf{X}$};
    \node [above right] at (Y) {\small $\mathsf{Y}$};
    \node [right] at (Z) {\small $\mathsf{Z}$};
    \node [above right] at (2Z) {\small $\mathsf{2Z}$};
    \node [above right] at (Y+Z) {\small $\mathsf{Y+Z}$};
    \node [below right] at (0) {\small $\mathsf{0}$};
    \node[right] at (0,1.1,2.25) {$G_{Z\rightleftharpoons 0}$};
    
    \draw[myarrow,transform canvas={xshift=-1pt,yshift=0pt}]  (Y) to node[xshift=0pt,yshift=-3pt] {} (X);
    \draw[myarrow,transform canvas={xshift=0pt,yshift=1pt}] (X) to node[xshift=0pt,yshift=-5pt] {} (Z);
    \draw[myarrow,transform canvas={xshift=1pt,yshift=0.5pt}] (2Z) to node[xshift=12pt,yshift=4pt] {} (Y+Z);
    \draw[myarrow2,transform canvas={xshift=1pt,yshift=0pt}] (0) to node[xshift=12pt,yshift=4pt] {} (Z);
    \draw[myarrow2,transform canvas={xshift=-1pt,yshift=0pt}] (Z) to node[xshift=12pt,yshift=0pt] {$H$} (0);
\end{axis}

\draw[rounded corners,lightgray] (current bounding box.south west) rectangle (current bounding box.north east);

\end{tikzpicture}
    \end{aligned}
    \end{align}
\end{example}

\begin{lemma}[ACR‐Reduction]\label[lemma]{lemma:ACRreduction}
  Let $G=(\X,\R,\C)$ be a network having ACR in a subset $\E\subset \X$ of species. If $G$ has $\ell$ positive steady states in some stoichiometric class, then the projection $G_{-\E}$ of $G$ over $\X\!\setminus\!\E$ has at least $\ell$ positive steady states in some stoichiometric class. That is,
  \[\cappos\bigl(G_{-\E}\bigr) \geq \cappos\bigl(G\bigr)\]
\end{lemma}

 Note that the projection of a steady state of $G$ over $\X\setminus\E$ does not need to be a steady state of $G_{-\E}$.\\
For Example: Take the Reaction Network with the reactions:   
    \begin{small}
    \begin{align}\label{eq:G_E}
    \begin{aligned}
    A + B \ce{<=>} C,\\
    B \ce{<=>} 0.
    \end{aligned}
    \end{align}
    \end{small}
    This has ACR in $B$, and for all rates having value 1 except the inflow to $B$ having a value 2, has a steady state with concentrations $x_A=T$, $x_B=$2 and $x_C=$2$T$ (The value of $T$ is based on the stoichiometric compatibility class). Now, this same steady state projected over $\{A,B,C\}\setminus B$ is not the steady state of $A \ce{<=>} C$ for rates being 1 (but the projected network would have at least one steady state in some stoichiometric class, namely one with equal concentration of $A$ and $C$).

\begin{proof}
    Without loss of generality, it is enough to show the statement is true when $\E$ only has one species (and iterate the procedure otherwise). Specifically, suppose that $\E = \{X_1\}$ and $G$ has ACR in $X_1$. Consider the modification $G_{-X_1}$ given by removing the species $X_1$. Suppose also that $G$ has $\ell$ positive steady states in some stoichiometric class. We want to show that $G_{-X_1}$ also does.
   
   Let $a$ be the ACR value of $X_1$ in $G$, i.e., any positive steady state $z = (z_1,\dots,z_n)$ of $G$ satisfies $z_1 = a$ (note that $a$ may depend on $\k$).  We will show that there is a choice of reaction rates $\k^*$ for $G_{-X_1}$, such that there is an injection between the steady states of $G$ in a certain stoichiometric class and the steady states of $G_{-X_1}$ in a certain stoichiometric class of the latter.

    Given a choice $\kappa^*$ of reaction rates for $G_{-X_1}$, let us write the associated mass action ODE system as
    \begin{align*}
        \dfrac{dx_i}{dt} &=\tilde{f_i}(x_2,\dots,x_n)\\
        &= \sum_{\tilde{y} \to \tilde{y}'\in \tilde{\R}} \kappa_{\tilde{y}\to \tilde{y}'}* x_2^{\tilde{y}_2}\cdots x_n^{\tilde{y}_n} (\tilde{y}_i' - \tilde{y}_i)
    \end{align*}
    for $i=2,\dots,n$, where $\tilde{\R}$ is the set of reactions of $G_{-X_1}$. Note that by construction any reaction $(y_2,\dots,y_n) \ce{->} (y_2',\dots,y_n')$ in $G_{-X_1}$ arises from a reaction $(y_1,\dots,y_n) \ce{->} (y_1',\dots,y_n')$ in $G$ where the species $X_1$ was removed. Even more, when seen as an E-graph, we can think of the vertices and edges (complexes and reactions) of $G_{-X_1}$ as a projection via $\pi:\mathbb{R}^{|\X|}\longrightarrow \mathbb{R}^{|\X\setminus \E|}$ of the vertices and edges of $G$ over the coordinates $i=2,\dots, n$. Note also that any reaction in $G$ projects to one of $G_{-X_1}$, although several reactions in $G$ can be mapped to the same one. Hence one can rewrite for $G$ the function $f_i$ as
     \begin{align*}
        f_i(x) &= \sum_{y \to y'\in \R} \kappa_{y\to y'} x_1^{y_1} x_2^{y_2}\cdots x_n^{y_n} (y_i' - y_i)\\
        &= \sum_{\tilde{y} \to \tilde{y}'\in \tilde{\R}} \left(\sum_{y\to y'\in \R: \pi(y)=\tilde{y}, \pi(y')=\tilde{y}'} \kappa_{y\to y'} x_1^{y_1} x_2^{y_2}\cdots x_n^{y_n} (y_i' - y_i)\right)\\
        &= \sum_{\tilde{y} \to \tilde{y}'\in \tilde{\R}} \left(\sum_{y\to y'\in \R: \pi(y)=\tilde{y}, \pi(y')=\tilde{y}'} \kappa_{y\to y'} x_1^{y_1} \right) x_2^{\tilde{y}_2}\cdots x_n^{\tilde{y}_n} (\tilde{y}'_i - \tilde{y}_i)
    \end{align*}
    
    for $i=2,\dots,n$, where $\pi((y_1,\dots,y_n))=(y_2,\dots,y_n)$ is the projection over the coordinates $i=2,\dots, n$.

   Now any positive steady state $z=(a,z_2,\ldots ,z_n)$ of $G$ satisfies:

    \begin{equation*}
        f_i(z)= \sum_{\tilde{y} \to \tilde{y}'\in \tilde{\R}} \left(\sum_{y\to y'\in \R: \pi(y)=\tilde{y},\newline \pi(y')=\tilde{y}'} \kappa_{y\to y'} a^{y_1} \right) z_2^{\tilde{y}_2}\cdots z_n^{\tilde{y}_n} (\tilde{y}'_i - \tilde{y}_i)=0
    \end{equation*}

    for $i=2,\dots,n$. But then, any such $z$ satisfies 
    $$  0= \sum_{\tilde{y} \to \tilde{y}'\in \tilde{\R}} \kappa^*_{\tilde{y}\to \tilde{y}'} z_2^{\tilde{y}_2}\cdots z_n^{\tilde{y}_n} (\tilde{y}'_i - \tilde{y}_i)= f_i'(z_2,\dots,z_n)$$
    for $i=2,\dots, n$, with reaction rates
    \begin{equation}\label{eq:k'}
        \kappa^*_{\tilde{y}\to \tilde{y}'} = \sum_{y\to y'\in \R: \pi(y)=\tilde{y}, \pi(y')=\tilde{y}'} \kappa_{y\to y'} a^{y_1},
    \end{equation}
and is therefore a positive steady state of $G_{-X_1}$ with reaction rates $\kappa^*$.
   Finally, recall from \cref{rmk:projected_stoich} that the stoichiometric subspace $\imG_{-X_1}$ of $G_{-X_1}$ is just the projection of the stoichiometric subspace $\imG$ of $G$.

    Now, suppose that $G$ has $\ell$ different positive steady states $z^j=(a,z^j_2,\dots,z^j_n)$, $j=1,\ldots ,l$ for some $\k$ in the same stoichiometric compatibility class. Then, these give rise to $\ell$ different  positive steady states $\tilde{z}^j=(z^j_2,\dots, z^j_n)$ of $G_{-X_1}$ for $\k^*$ obtained from $\k$ and $a$ as in \eqref{eq:k'}, which will also be in the same compatibility class of $G_{-X_1}$: indeed, for any $j,k\in\{1,\dots,l\}$, $j\neq k$, as $z^j$ and $z^k$ are in the same compatibility class, by \cref{rmk:same_compatibility} we have that $z^k - z^j = (0,z^k_2 - z^j_2,\dots,z^k_n-z^j_n) \in \imG$. But then $\tilde{z}^k  -\tilde{z}^j = (z^k_2 - z^j_2,\dots,z^k_n-z^j_n)\in \imG_{-X_1}$ the stoichiometric subspace of $G_{-X_1}$, so one can prove by pairs that all projected steady states $\tilde{z}^j$ are in the same compatibility class. 
\end{proof}

We are now in a position to prove \Cref{thm:enzyme-open-mono}.

\begin{proof}(of \Cref{thm:enzyme-open-mono}) If $\operatorname{cap}_{\mathrm{pos}}(G_{\E\rightleftharpoons0})=0$, the inequality is immediate. Otherwise, fix any positive rate constants for which $G_{\E\rightleftharpoons0}$ has a positive steady state. \Cref{thm:open-is-ACR} guarantees ACR in all species of $\E$. Thus, \Cref{lemma:ACRreduction} applies to $G_{\E\rightleftharpoons0}$ and $\E$. Its proof gives an injective projection of the positive steady states in any one stoichiometric compatibility class into the positive steady states of $(G_{\E\rightleftharpoons0})_{-\E}=G_{-\E}$ in a single compatibility class, for one induced choice of positive rate constants. This also applies when there are infinitely many steady states. Maximizing over the rate constants and compatibility classes yields the claimed capacity inequality.
    
\end{proof}

\begin{example}\label[example]{ex:ACRreduction}
    Consider $G$ given by \eqref{ex:E-graph}. We will prove that $G_{Z\rightleftharpoons 0}$ (see \eqref{tikz:open}) is monostationary by using the previous results and showing that $G_{-Z}$ is monostationary.

    As we saw in \cref{ex:open-is-ACR}, $G_{Z\rightleftharpoons0}$ has ACR in $Z$, and there are no conservation laws, so we can apply \cref{lemma:ACRreduction} to $G_{Z\rightleftharpoons0}$. Note that $(G_{Z\rightleftharpoons0})_{-Z}$ gives:
    \begin{small}
    \begin{align}\label{eq:G_Z-ACRreduction}
    \begin{aligned}
    Y \ce{->} X \ce{->} 0,\\
    0 \ce{->} Y,
    \end{aligned}
    \end{align}
    \end{small}
    which is the same reaction network as $G_{-Z}$ (see \eqref{eq:G_E}) as shown in \cref{ex:projection}. Because $G_{-Z}$ is monomolecular, it has deficiency zero and thus it is monostationary (as shown in \cref{rmk:monomolecular}). Moreover, as it is weakly reversible and has no conservation laws, there is a unique positive steady state (see \Cref{thm:DefZero}). Thus, \Cref{lemma:ACRreduction} implies that $G_{Z\rightleftharpoons 0}$ is monostationary too.
\end{example}



More generally, instead of adding inflow and outflow reactions for the species in $\E$, we may couple $G=(\X,\C,\R)$ to a network $H=(\E,\C',\R')$. Suppose that $\E\subset\X$ is independently conserved in $G$ and that $H$ has full stoichiometric dimension (no conservation laws) and a unique positive steady state. Whenever $G\cup H$ has a positive steady state, \cref{lemma:SumSteadyStates} implies that its concentrations on $\E$ coincide with that unique steady state of $H$, yielding ACR in those species. Since $(G\cup H)_{-\E}=G_{-\E}$, \cref{lemma:ACRreduction} then bounds the positive steady-state capacity of $G\cup H$ by that of $G_{-\E}$. If $G\cup H$ has no positive steady states, the same bound holds trivially. This gives the following proposition:


\begin{proposition}
    Let $G=(\X,\C,\R)$ be a network and let $\E=\{X_i\}_{i=1,\dots,k}\subset \X$ be a set of independently conserved species. Let $H=(\E,\C',\R')$ be a network with $\E$ as its set of species such that it has full stoichiometric dimension (no conservation laws) and a unique positive steady state.
    Then
    \[\cappos\bigl(G_{-\E}\bigr) \geq \cappos\bigl(G\cup H\bigr)\]
\end{proposition}

\begin{proof}
    Let $z_H$ be the unique positive steady state of the network $H$. First, note that this is trivially true if $G\cup H$ has no positive steady state, as then $\cappos\bigl(G\cup H\bigr) =0$, so we will assume that there exists a positive steady state.
    Then by \cref{lemma:SumSteadyStates}, any steady state $z\in\mathbb{R}^n_{>0}$ of $G\cup H$ must satisfy that its projection gives $z_{\E} = z_H$, and thus $G\cup H$ has ACR in $\E$. Then, as $(G\cup H)_{-\E} = G_{-\E}$, we have by \cref{lemma:ACRreduction} that
    \[\cappos\bigl(G_{-\E}\bigr) = \cappos\bigl(\bigl(G\cup H\bigr)_{-\E}\bigr) \geq \cappos\bigl(G\cup H\bigr)\]
\end{proof}

Here is an application of this generalization: in the following section, whenever we prove that opening the enzymes gives a monostationary network, the result still holds if instead we add an arbitrary full dimensional network $H$ on the enzymes such that $H$ has a unique positive steady state. For example, instead of $H = \{E\rightleftharpoons 0, F\rightleftharpoons 0\}$, we can add $H = \{E\rightarrow 0, 0 \rightarrow F, E \rightleftharpoons F\}$.

\medskip

\subsection{Application to the $n$‐site double phosphorylation cycle}\label{sec:n-site_monost}

\begin{theorem}\label{thm:enzyme-open-nsite}
  The $n$‐site double phosphorylation cycle \eqref{eq:n-site} with the enzymes $\E=\{E,F\}$ open, $\mathcal{P}^n_{\E\rightleftharpoons 0}$, is monostationary.
  Equivalently,
  \[
  \cappos\bigl(\mathcal{P}^n_{\E\rightleftharpoons0}\bigr)\leq 1.
  \]
\end{theorem}

\begin{proof}
    First, note that $E$ and $F$ are independently conserved in $\mathcal{P}^n$. Indeed, $E$ appears in conservation law $L_E$ but not in $L_F$ (see \eqref{eq:cl_n-site}), and similarly $F$ appears in $L_F$ but not in $L_E$.

    Now consider the network $\mathcal{P}^n_{-\{E,F\}}$: we will prove that it is monostationary. Note that any bimolecular complex in $\mathcal{P}^n$ has either the species $E$ or $F$ in it, so $\mathcal{P}^n_{-\{E,F\}}$ will be monomolecular. Moreover, the number of complexes in $\mathcal{P}^n_{-\{E,F\}}$ equals the number of species, which is $3n+1$. Note also that $\mathcal{P}^n_{-\{E,F\}}$ is weakly reversible, as there is a cycle that goes through every complex at least once, and it has one connected component, so $\ell=1$:

\[\begin{tikzcd}
	& {ES_0} && \dots && {ES_{n-1}} \\
	{S_0} && {S_1} & \dots & {S_{n-1}} && {S_n} \\
	& {FS_1} && \dots && {FS_n}
	\arrow[shift left, from=1-2, to=2-1]
	\arrow[from=1-2, to=2-3]
	\arrow[shift left, from=1-4, to=2-3]
	\arrow[from=1-4, to=2-5]
	\arrow[shift left, from=1-6, to=2-5]
	\arrow[from=1-6, to=2-7]
	\arrow[shift left=3, from=2-1, to=1-2]
	\arrow[shift left=3, from=2-3, to=1-4]
	\arrow[shift left=3, from=2-3, to=3-2]
	\arrow[shift left=3, from=2-5, to=1-6]
	\arrow[shift left=3, from=2-5, to=3-4]
	\arrow[shift left=3, from=2-7, to=3-6]
	\arrow[from=3-2, to=2-1]
	\arrow[shift left, from=3-2, to=2-3]
	\arrow[from=3-4, to=2-3]
	\arrow[shift left, from=3-4, to=2-5]
	\arrow[from=3-6, to=2-5]
	\arrow[shift left, from=3-6, to=2-7]
\end{tikzcd}\]
     Finally, the dimension of the stoichiometric subspace is $\dim(\imG) = \ker(W)=3n$, as there are $3n+1$ species and only one conservation law remaining. Finally, the deficiency of $\mathcal{P}^n_{-\{E,F\}}$ is $\delta = |\C| - \ell - \dim(\imG) = 3n+1 - 1 -3n = 0$, and by the Deficiency Zero Theorem (\Cref{thm:DefZero}), $\mathcal{P}^n_{-\{E,F\}}$ is monostationary. Thus, by \Cref{thm:enzyme-open-mono}, so is $\mathcal{P}^n_{\{E,F\}\rightleftharpoons 0}$.
\end{proof}

\begin{remark}
     A similar proof shows that opening $\E = \{E, F, S_i\}$ for any $i$ is also monostationary, using the third conservation law. Indeed, in this case the number of species is $3n$, the number of conservation laws is 0, and the number of complexes is again $3n+1$, as the complex for $S_i$ is lost, but we get $0$ as a new complex. Thus, the deficiency is still $\delta = 3n+1 - 1 - 3n = 0$.
     
     In this case, the stoichiometric subspace is already the whole space $\mathbb{R}^{3n}_{\geq 0}$, and the same will be true if $\E$ is given by both enzymes and any subset of substrates, due to \Cref{thm:enlarging_preserving_stoichiometry}.

     In any of these cases the network is monostationary.
\end{remark}


\begin{remark}
    The proof doesn't depend on the structure of the intermediates added. Consider the core model $G$ of the $n$-site double phosphorylation cycle as defined in \Cref{extension-model} given by collapsing all the intermediates:
    \begin{align}\label{eq:core-n-site}
    S_0 + E \ce{->} S_1 + E \ce{->}  \dots \ce{->}  & S_{i}+E  \ce{->} \dots \ce{->}  S_n+E \nonumber\\
    S_n + F \ce{->} S_{n-1} + F \ce{->}  \dots \ce{->}  & S_{n-i}+F  \ce{->} \dots \ce{->}  S_0+F  \\
    \end{align}
    Then any extension model as defined in \Cref{extension-model} with the enzymes $\E=\{E,F\}$ open will be monostationary by the same proof structure. Indeed, note that $G_{-E,F}$ is still weakly reversible with deficiency zero:
    \begin{align}\label{eq:core-n-site}
    S_0 \ce{<=>} S_1 \ce{<=>}  \dots \ce{<=>}S_{i}\ce{<=>} \dots \ce{<=>}  S_n
    \end{align}
    As there is only one linkage class, adding intermediates will not change this, and as the intermediate complexes added are all monomolecular it will not change the deficiency. Note that the extension with intermediates will keep reaction paths between core complexes, even after taking the projection. Thus for any extended network $G'$ with intermediates $\mathcal{Y}$, we have that $G'_{-E,F}$ is monostationary.
\end{remark}

\begin{remark}
    
 It is not clear from any theoretical development what would happen to $\mathcal{P}^2_{\E\rightleftharpoons 0}$ with more complicated choices of $\E$. For example, from numerical computations we know that if $\E=\{ E,S_1\}$ or $\E=\{E,S_1,S_2\}$ the network is monostationary, while for $\E=\{ E,S_0\}$ or even $\E=\{ E,S_0,S_1\}$ it admits multistationarity. The results in this work are also not enough to prove any of these cases.
\end{remark}

In \Cref{sec:ACR} we introduced the property of ACR in a species. It comes as a consequence  of \Cref{thm:open-is-ACR} that the $n$-site double phosphorylation cycle gains ACR in any enzyme when that enzyme is open:

\begin{corollary}
    Let $n\geq 1$ and let $\E$ be any subset of enzymes. Then $\mathcal{P}^n_{\E\rightleftharpoons 0}$ has ACR in all species of $\E$.
\end{corollary}

\begin{proof}
    Direct from \Cref{thm:open-is-ACR} using the fact that $\mathcal{P}^n$ has a positive steady state for $n\geq 1$ and that $E$ and $F$ are independently conserved in $\mathcal{P}^n$.
\end{proof}

\subsection{Extension to cascade networks}\label{sec:cascade_monost}

This approach can be greatly generalized to work on cascade-type networks. These networks have species that act as both \textit{substrates} and \textit{enzymes} in different reactions. For example, consider the following simple cascade diagram:

\[\begin{tikzcd}
	W && {W^*} \\
	\\
	& Z && {Z^*}
	\arrow["{E_1}", curve={height=-18pt}, from=1-1, to=1-3]
	\arrow["{E_2}", curve={height=-18pt}, from=1-3, to=1-1]
	\arrow[""{name=0, anchor=center, inner sep=0}, curve={height=-18pt}, from=3-2, to=3-4]
	\arrow["{E_3}", curve={height=-18pt}, from=3-4, to=3-2]
	\arrow[from=1-3, to=0]
\end{tikzcd}\]

which gives the following reaction network $G$:
\begin{small}
\begin{align}\label{eq:cascade-1}
\begin{split}
W + E_1 &\ce{<=>} WE_1 \ce{->} W^* + E_1 \\
W^* + E_2 &\ce{<=>} W^*E_2 \ce{->} W + E_2 \\
Z + W^* &\ce{<=>} ZW^* \ce{->} Z^* + W^* \\
Z^* + E_3 &\ce{<=>} Z^*E_3 \ce{->} Z + E_3 
\end{split}
\end{align}
\end{small}

There are 5 independent conservation laws:
\begin{align}
        L_{E_1}: & \quad x_{E_1} +x_{WE_1} = T_{E_1}\notag \\
        L_{E_2}: & \quad x_{E_2} +x_{W^*E_2} = T_{E_2}\notag \\
        L_{E_3}: & \quad x_{E_3} +x_{Z^*E_3} = T_{E_3} \\
        L_{W}: & \quad x_W + x_{WE_1} + x_{W^*} + x_{W^*E_2} + x_{ZW^*} = T_W\notag \\
        L_{Z}: & \quad x_Z + x_{ZW^*} + x_{Z^*} + x_{Z^*E_3} = T_Z\notag 
\end{align}
    where $T=(T_{E_1}, T_{E_2}, T_{E_3}, T_W, T_Z)$ takes values in $\mathbb{R}^5_{>0}$.

\begin{proposition}\label[proposition]{prop:enzyme-open-cascade}
The cascade network $G$ (\Cref{eq:cascade-1}), with the enzymes $\E = \{E_1, E_2, E_3, W^*\}$ open, $G_{\E\rightleftharpoons0}$, is monostationary.
\end{proposition}
\begin{proof}
    Note that $\E$ is independently conserved, as each enzyme shows up in exactly one of the conservation laws $L_{E_1}, L_{E_2}, L_{E_3}, L_W$. 

    We will again prove that $G_{-\E}$ is monostationary, and then apply \Cref{thm:enzyme-open-mono}. But this network,

\[\begin{tikzcd}
	& {WE_1} &&& {ZW^*} \\
	W && 0 & Z && {Z^*} \\
	& {W^*E_2} &&& {Z^*E_3}
	\arrow[shift left, from=1-2, to=2-1]
	\arrow[from=1-2, to=2-3]
	\arrow[shift left, from=1-5, to=2-4]
	\arrow[from=1-5, to=2-6]
	\arrow[shift left=3, from=2-1, to=1-2]
	\arrow[shift left, from=2-3, to=3-2]
	\arrow[shift left=3, from=2-4, to=1-5]
	\arrow[shift left, from=2-6, to=3-5]
	\arrow[from=3-2, to=2-1]
	\arrow[shift left=3, from=3-2, to=2-3]
	\arrow[from=3-5, to=2-4]
	\arrow[shift left=3, from=3-5, to=2-6]
\end{tikzcd}\]
    has deficiency $\delta = |\C| - \ell - \dim \imG_{-\E} = 8 - 2 - (7-1) = 0$, where $\dim \imG_{-\E} = 7 - 1$ as there are 7 species and one conservation law left. Then, by \Cref{thm:DefZero}, $G_{-\
    E}$ is monostationary, and by \Cref{thm:enzyme-open-mono}, so is $G_{\E\rightleftharpoons 0}$.
\end{proof}

To showcase how our method can be used to study more biologically meaningful networks, we will use the example of the mitogen-activated protein kinase (MAPK) cascade. The MAPK cascade (\cref{fig:MAPK-Cascade-1}) is a heavily studied network that models multiple biological behaviors in cells. The model was first developed in \cite{HuangFerrell1996}, where they showed that this cascade may exhibit ultrasensitive behavior, and thus it is a good candidate for switch-like processes like mitogenesis or cell-fate determination. The model has been numerically shown to exhibit bistability and oscillatory behavior \cite{Qiao2007,Ferrell1998,Kholodenko2000}, and proven in \cite{Banaji2018} to exhibit nondegenerate multistationarity. Here we will show how our results prove that when opening all species acting like enzymes the network becomes monostationary.

\begin{figure}[h]
    \centering
    \includegraphics[width=0.5\linewidth]{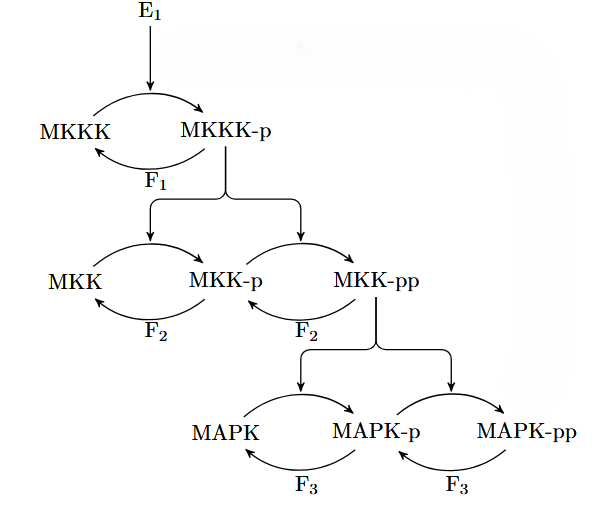}
    \caption{The MAPK cascade network}
    \label{fig:MAPK-Cascade-1}
\end{figure}

As in \cite{Banaji2018}, we use the following abbreviations for readability reasons of the reaction networks: $X = \text{MAPK}, Y = \text{MKK}, Z = \text{MKKK}.$ We call the MAPK network $\mathcal{M}^1$ (\cref{fig:MAPK-Cascade-1}). This is a network with 22 species and 30 reactions, which is represented by the following reactions:

\begin{small}
\begin{align}\label{eq:MAPK-Cascade-1}
\begin{split}
E_1 + Z \ce{<=>} E_1Z \ce{->} E_1 + Z_P, \quad F_1 + Z_P \ce{<=>} F_1Z_P \ce{->} F_1 + Z,\\
Z_P + Y \ce{<=>} Z_PY \ce{->} Z_P + Y_P \ce{<=>} Z_P Y_P \ce{->} Z_P + Y_{PP},\\
F_2 + Y_{PP} \ce{<=>} F_2Y_{PP} \ce{->} F_2 + Y_P \ce{<=>} F_2 Y_P \ce{->} F_2 + Y,\\
Y_{PP} + X \ce{<=>} Y_{PP}X \ce{->} Y_{PP} + X_P \ce{<=>} Y_{PP}X_P \ce{->} Y_{PP} + X_{PP},\\
F_3 + X_{PP} \ce{<=>} F_3 X_{PP} \ce{->} F_3 + X_P \ce{<=>} F_3 X_{P} \ce{->} F_3 + X.
\end{split}
\end{align}
\end{small}

\begin{proposition}\label[proposition]{prop:enzyme-open-cascade-1}
    Let $\mathcal{M}^1_{\E\rightleftharpoons0}$ be the network obtained from the MAPK cascade, $\mathcal{M}^1$ (\Cref{eq:MAPK-Cascade-1}) by opening the species $\E = \{F_1, F_2, F_3, E_1, Z_P, Y_{PP}\}$. Then $\mathcal{M}^1_{\E\rightleftharpoons0}$ is monostationary.
\end{proposition}

\begin{proof}
    Note that $\mathcal{M}^1$ has 7 conservation laws, for the total amounts of $X, Y, Z, E_1, F_1, F_2, F_3$, (\cite{HuangFerrell1996}) and that the set of the species behaving like enzymes, $\E = \{F_1, F_2, F_3, E_1, Z_P, Y_{PP}\}$ is independently conserved in $\mathcal{M}^1$. Thus, by \Cref{thm:enzyme-open-mono}, it is enough to prove that the projected network $\mathcal{M}^1_{-\E}=(\mathcal{M}^1_{\E\rightleftharpoons 0})_{-\E}$ is monostationary to show that so is $\mathcal{M}^1_{\E\rightleftharpoons 0}$. Now $\mathcal{M}^1_{-\E}$ is given by:

\[\begin{tikzcd}
	& {E_1Z} && {Z_PY_P} && {Z_PY} \\
	Z && 0 && {Y_P} && Y \\
	& {F_1Z_P} && {F_2Y_{PP}} && {F_2Y_P} \\
	& {Y_{PP}X} && {Y_{PP}X_P} \\
	X && {X_P} && {X_{PP}} \\
	& {F_3X_P} && {F_3X_{PP}}
	\arrow[shift left, from=1-2, to=2-1]
	\arrow[from=1-2, to=2-3]
	\arrow[from=1-4, to=2-3]
	\arrow[shift left=3, from=1-4, to=2-5]
	\arrow[from=1-6, to=2-5]
	\arrow[shift left=3, from=1-6, to=2-7]
	\arrow[shift left=3, from=2-1, to=1-2]
	\arrow[shift left, from=2-3, to=3-2]
	\arrow[shift left, from=2-3, to=3-4]
	\arrow[shift left, from=2-5, to=1-4]
	\arrow[shift left, from=2-5, to=3-6]
	\arrow[shift left, from=2-7, to=1-6]
	\arrow[from=3-2, to=2-1]
	\arrow[shift left=3, from=3-2, to=2-3]
	\arrow[shift left=3, from=3-4, to=2-3]
	\arrow[from=3-4, to=2-5]
	\arrow[shift left=3, from=3-6, to=2-5]
	\arrow[from=3-6, to=2-7]
	\arrow[shift left, from=4-2, to=5-1]
	\arrow[from=4-2, to=5-3]
	\arrow[shift left, from=4-4, to=5-3]
	\arrow[from=4-4, to=5-5]
	\arrow[shift left=3, from=5-1, to=4-2]
	\arrow[shift left=3, from=5-3, to=4-4]
	\arrow[shift left, from=5-3, to=6-2]
	\arrow[shift left, from=5-5, to=6-4]
	\arrow[from=6-2, to=5-1]
	\arrow[shift left=3, from=6-2, to=5-3]
	\arrow[from=6-4, to=5-3]
	\arrow[shift left=3, from=6-4, to=5-5]
\end{tikzcd}\]

 which has deficiency $\delta = |\C| - \ell - \dim \imG_{-\E} = 17 - 2 - (16-1) = 0$, where $\dim \imG_{-\E} = 16 - 1$, as there are 16 species and one conservation law left. Then, by \Cref{thm:DefZero}, $\mathcal{M}^1_{-\E}$ is monostationary and, by \Cref{thm:enzyme-open-mono}, so is $\mathcal{M}^1_{\E\rightleftharpoons 0}$.
\end{proof}

\section{Discussion and Future Work}

Our results give a compact, structural picture of how semi-opening a phosphorylation–dephosphorylation cycle reshapes its steady-state landscape. Working under mass action kinetics we show that these networks are remarkably asymmetric with respect to which species are allowed to exchange with an external reservoir: permitting any nonempty set of substrates to flow in and out preserves the network’s capacity for (nondegenerate) multistationarity, while permitting the enzymes to flow (even when some substrates are also opened) destroys that capacity and yields a monostationary system. The two claims are proved by different but complementary means. The substrate result is constructive: steady states and their nondegeneracy extend inductively as phosphorylation sites are added, so adding more phosphorylation sites does not remove the algebraic mechanism by which multiple positive equilibria arise, even when some (or all) substrates are open. By contrast, the enzyme result is structural: opening an independently conserved enzyme subsystem produces ACR in those species, and a projection over the remaining species often reduces the dynamics to a monostationary core (for example a monomolecular, weakly reversible projection where we can apply deficiency-zero arguments). Thus ACR detection together with projection gives a simple criterion (\cref{thm:enzyme-open-mono}) that explains why enzyme exchange acts as a robust “switch” from multistationarity to uniqueness of equilibrium. As an additional result we obtain here that the processive $n$-site double phosphorylation cycle with both enzymes open has ACR in these, so there is certain robustness in their concentration too.

Biologically, this dichotomy has an intuitive interpretation. Substrates are the carriers of modification states and can be coupled to downstream processes without losing their ability to support multiple steady states; their exchange models consumption, synthesis, or wiring to other pathways, and so preserves the cycle’s switching potential. Enzymes, on the other hand, are the scarce, conserved mediators whose conservation and reuse create the competition that produces multiple steady states; allowing enzymes to be replenished or removed from the local system breaks that conservation and collapses the mechanism for multistationarity. This perspective complements and sharpens earlier observations about fully open networks and entrapped species: it identifies which classes of species matter most for switching and why.

The paper also contributes methodological tools of independent interest. The inductive construction for extending steady states across sites is elementary and broadly applicable to multisite motifs; the ACR-plus-projection reduction ties robustness properties of a small, open subsystem to global conclusions via classical network theorems.  Although the method we use here to conclude that the simpler networks are monostationarity rely on deficiency theory, one could find other simplifications, in the sense described in \Cref{sec:overview_monost}, that allow to establish results for the network of interest based on the knowledge (maybe of different nature, not related to deficiency) of simpler networks, by carefully choosing an independently conserved subset of species.    These tools are relatively easy to check in enzyme-centric architectures and thus provide a practical route to diagnose when semi-open modifications will preserve or eliminate multistationarity.

In short, by isolating the different roles of substrates and enzymes and by connecting ACR and projection to classical deficiency arguments, this work provides a compact, testable explanation for when environmental exchange preserves biochemical switching and when it enforces unique steady states, a distinction with direct implications for modeling, experimental design, and synthetic circuit engineering.

This work has clear limitations: our hypotheses include an independent-conservation condition for the open enzymes, and all results are proven with limited claims about stability. Future work will aim at combining our structural results with bifurcation and stability analyses. But first we will explore the general use of these tools in future work and apply the framework to larger cascade and phosphorylation families and experimentally grounded models. We will also study how to best use these new tools together with the extensive tools already found in the literature in the more general context of embedded networks and inheritance results.

\section*{Acknowledgments}

We acknowledge the Max Planck Institute for the Mathematical Sciences in Leipzig, Germany as well as the Simons Laufer Mathematical Sciences Institute for funding and organizing the MSRI-MPI Leipzig Summer Graduate School on Algebraic Methods for Biochemical Reaction Networks in June 2023, where the initial steps of this research were taken, as well as the instructors of the school, and in particular Elisenda Feliu for suggesting these questions and guiding us in the first steps.

PN would like to thank Dr. Philippe Nghe for constructive discussion and funding through Grant ERC AbioEvo (101002075).
BP-E has been funded by the Spanish Ministry of Economy project with reference number PID2022-138916NB-I00.
DR was supported in part by the Fulbright Program. 

\textbf{Data availability} Data sharing not applicable to this article as no datasets were generated or analysed during the current study.

\section*{Declarations}
\textbf{Conflict of Interest} The authors declare that they have no Conflict of interest.\\

\bibliographystyle{alpha} 
\bibliography{references}

\end{document}